%% file: 0-main.tex
\documentclass[12pt,reqno]{amsart}
\usepackage{mysty}

\title[An explicit Chebotarev variant of the Brun--Titchmarsh theorem]{Modular forms and an explicit Chebotarev variant of the Brun--Titchmarsh theorem}

\author{Daniel Hu}
\address{Department of Mathematics, Princeton University}
\email{danielhu@princeton.edu}

\author{Hari R. Iyer}
\address{Department of Mathematics, Harvard University}
\email{hiyer@college.harvard.edu}

\author{Alexander Shashkov}
\address{Department of Mathematics, Williams College}
\email{aes7@williams.edu}



\date{\today}

\begin{document}

\maketitle

\begin{abstract}
    We prove an explicit Chebotarev variant of the Brun--Titchmarsh theorem. This leads to explicit versions of the best known unconditional upper bounds toward conjectures of Lang and Trotter for the coefficients of holomorphic cuspidal newforms. In particular, we prove that
    \begin{equation*}
        \lim_{x \to \infty} \frac{\#\{1 \leq n \leq x \mid \tau(n) \neq 0\}}{x} > 1-1.15 \times 10^{-12},
    \end{equation*}
    where $\tau(n)$ is Ramanujan's tau-function. This is the first known positive unconditional lower bound for the proportion of positive integers $n$ such that $\tau(n) \neq 0$. 
\end{abstract}

\setcounter{tocdepth}{1}
\tableofcontents

\section{Introduction}\label{sec:1}
\input{1-intro.tex}

\section{Preliminaries}\label{sec:2}
\input{2-notation}

\input{2.5-reduceToAbelian}

\section{Sums over integral ideals}\label{sec:3}
\input{3-aux_estimates}

\section{Implementing the Selberg sieve}\label{sec:4}
\input{4-selberg-sieve}

\section{Brun--Titchmarsh for abelian extensions}\label{sec: BT}
\input{5-BT}


\section{Upper bounds for the Lang--Trotter conjecture}\label{sec: LT}
\input{7-LT}

\section{Proof of Theorems \ref{thm: LT tau bound} and \ref{thm: Lehmer+}} \label{sec: nonzero}
\input{8-LT-applications}

\section{Proof of Theorems \ref{thm: level 2 bound} and \ref{thm: level 2}}\label{sec: elliptic curves}
\input{9-elliptic-curve}

\bibliographystyle{alpha}
\bibliography{bib}
\nocite{*}

\end{document}

%% file: 1-intro.tex
Let $\pi(x; q, a)$ be the number of primes $p \leq x$ such that $p \equiv a \Mod{q}$. The prime number theorem for primes in arithmetic progressions states that if $\gcd(a, q) = 1$, then for any $N > 0$, there exists an (ineffective) constant $C_N > 0$ such that if $q \leq (\log x)^N$, then
\begin{equation}
    \Big|\pi(x; q, a) - \frac{\Li(x)}{\varphi(q)} \Big| \ll_N x \exp(-C_N \sqrt{\log x})),
\end{equation}
where $\Li(x)$ is the logarithmic integral
\begin{equation*}
    \Li(x) = \int_2^x \frac{dt}{\log t} \sim \frac{x}{\log x}.
\end{equation*}
Under the generalized Riemann hypothesis (GRH) for Dirichlet $L$-functions, we can extend the range to $q \leq x^{1/2}(\log x)^{-1}$, and reduce the error term to $O(x^{1/2}\log x)$. For many applications, providing a better range for $x$ is of crucial importance, though this may require the tradeoff that we work with upper or lower bounds for $\pi(x; q, a)$ in lieu of an asymptotic. For instance, the Brun--Titchmarsh theorem in the form proved by Montgomery and Vaughan~\cite{MV} states that if $q < x$, then
\begin{equation}\label{eq: BT original}
    \pi(x; q, a) \leq \frac{2x}{\varphi(q)\log(x/q)}.
\end{equation}

We consider analogous results over number fields. Let $L/F$ be a Galois extension of number fields with Galois group $G$. For each prime ideal $\mathfrak{p}$ of $F$ which does not ramify in $L$, we use the Artin symbol $\big(\frac{L/F}{ \mathfrak{p} } \big)$ to denote the conjugacy class in $G$ of Frobenius elements of primes in $L$ lying over $\mathfrak{p}$. Given a fixed conjugacy class $C$ of $G$, define
\begin{equation}\label{eq: piC def}
    \pi_C(x,L/F) \coloneqq \#\Big\{ \mathfrak{p} \text{ unramified in } L,~ \Big(\frac{L/F}{ \mathfrak{p} } \Big) = C,~\N_{F/\mathbb{Q}}\mathfrak{p} \leq x\Big\}.
\end{equation}
The Chebotarev density theorem states that
\begin{equation}\label{eq: CDT}
    \lim_{x \to \infty} \frac{\pi_C(x, L/F)}{\Li(x)} = \frac{|C|}{|G|}.
\end{equation}
The first asymptotic version of this result with an effective error term was obtained by Lagarias and Odlyzko~\cite{LO}. They showed that if $x \geq \exp(10 [L:\Q](\log D_L)^2)$, then there exist absolute, effectively computable constants $c_0 , c_1 > 0$ such that
\begin{equation}\label{eq: LO}
    \Big|\pi_C(x, L/F) - \frac{|C|}{|G|}\Li(x) \Big| \le \frac{|C|}{|G|} \Li(x^{\beta_0}) + c_0x\exp(-c_1[L:\QQ]^{-1/2} (\log x)^{1/2}),
\end{equation}
where the $\beta_0$ term is only present when the Dedekind zeta function $\zeta_L(s)$ has a Siegel zero at $\beta_0$. 
Lagarias, Montgomery, and Odlyzko \cite{montgomery1979bound} improved the range at the cost of only proving an upper bound for $\pi_C(x, L/F)$. They showed that there exist absolute, effectively computable constants $c_2, c_3 > 0$ such that
\begin{equation}\label{eq: BT chebotarev bad}
    \pi_C(x, L/F) \le c_2 \frac{|C|}{|G|} \Li(x) , \qquad \log x \ge c_3 (\log D_L)(\log \log D_L)(\log \log \log e^{20} D_L).
\end{equation}
Refinements to the effective version of the Chebotarev density theorem were later made by Serre \cite{Ser81} and V. K. Murty \cite{Mur97}.

In a series of papers \cite{thorner2017explicit, TZ, thorner2019unified}, Thorner and Zaman improved the previous results on the asymptotic form and upper and lower bounds for the Chebotarev density theorem. Here we focus on their analogue of the Brun--Titchmarsh theorem \eqref{eq: BT original}. Let $A \subseteq G$ be an abelian subgroup such that $C \cap A$ is nonempty, and let $K$ be the fixed field of $A$. Set $n_K = [K:\Q]$. Let $\widehat{A}$ be the dual group, $\mathfrak{f}_\chi$ the conductor of $\chi \in \widehat{A}$, and define
\begin{equation}
    \mathcal{Q} = \mathcal{Q}(L/K) \coloneqq \max_{\chi \in \widehat{A}} \N_{K/\mathbb{Q}}\mathfrak{f}_{\chi}.
\end{equation}
Thorner and Zaman \cite[Theorem 1.1]{TZ} proved that there exists an absolute, effectively computable constant $c_4 > 0$ such that if
\begin{equation}\label{eq: TZ-range}
    x \ge c_4 \Big(D_K^{164} \mathcal{Q}^{123} + D_K^{55} \mathcal{Q}^{87} n_K^{68 n_K} + D_K^{2} \mathcal{Q}^{2} n_K^{14,000 n_K} \Big),
\end{equation}
then the bound on $\pi_C(x,L/F)$ in $\eqref{eq: BT chebotarev bad}$ holds.
We have that $D_K\mathcal{Q} \le D_L$ by the conductor-discriminant formula, and by Minkowski's bound, there exists an absolute, effectively computable constant $c_5 > 0$ such that $n_K \le c_5\log D_K \le c_5\log D_L$. Thus, there exists an absolute, effectively computable constant $c_6 > 0$ such that \eqref{eq: TZ-range} holds when $\log x \ge c_6 (\log D_L) (\log \log D_L)$, so \eqref{eq: TZ-range} is a uniform improvement over the range in \eqref{eq: BT chebotarev bad}. There also exists an absolute, effectively computable constant $c_7 > 0$ such that for most number fields $K$, $n_K \le c_7 (\log D_K)/\log\log D_K$. In this case, 
there exists an absolute, effectively computable constant $c_8 > 0$ such that \eqref{eq: TZ-range} holds when $\log x \ge c_8 \log D_L$.


Our main result is a completely explicit Chebotarev variant of the Brun--Titchmarsh theorem with a range of $x$ that improves upon \eqref{eq: TZ-range}. 
\begin{thm}\label{thm: BT general}
Let $L/F$ be a Galois extension of number fields with Galois group $G$, and let $C \subset G$ be a conjugacy class. Let $A$ be an abelian subgroup of $G$ such that $C \cap A$ is nonempty, and let $K$ be the fixed field of $A$. If
\begin{equation}
    x \ge e^{92n_K+36}(D_K\mathcal{Q})^{8.4} n_K^{4.2n_K},
\end{equation}
then
\begin{equation}
    \pi_C(x, L/F) \le 11.3 \frac{|C|}{|G|} \frac{x}{\log x}.
\end{equation}
\end{thm}
\begin{remark}
Note that $\frac{x}{\log x} \le \Li(x)$ for all $x \ge 10$.
\end{remark}
\begin{remark}
In the case where $F = K = \mathbb{Q}$ and $L = \mathbb{Q}(e^{2\pi i/q})$, we recover a weakened leading constant in the classical Brun--Titchmarsh theorem, valid in the range $x \geq e^{128} q^{8.4}$.
\end{remark}

To prove \cref{thm: BT general}, we use the Selberg sieve over number fields as in \cite{Wei}. 
In order to implement this strategy, we obtain estimates for the number of integral ideals in $K$ with norm less than $x$ satisfying certain congruence conditions. We prove these by bounding contour integrals of Hecke $L$-functions smoothed by a test function. This input may alternatively be treated by translation to a lattice point-counting argument, as in work of Debaene~\cite{Deb19}, who obtained a similar result with better exponents of $D_K$ and $\mathcal{Q}$ but worse dependence on $n_K$. On the other hand, Debaene only works with the base field $F = \mathbb{Q}$, whereas the applications which motivate our work involve results over a general base field $F$. Moreover, the $n_K$-dependence is crucial for our applications. 

We apply \cref{thm: BT general} to study values of the Fourier coefficients of modular forms. For $\Im(z) > 0$, let
\begin{equation*}
    f(z) = \sum_{n=1}^\infty a_f(n) e^{2\pi inz} \in \mathcal{S}_k^{\text{new}}(\Gamma_0(N))
\end{equation*}
be a normalized cusp form of even weight $k \geq 2$ which is a newform of level $N = N_f$ with trivial nebentypus and without complex multiplication.
Fix $a \in \mathbb{Z}$, and define
\begin{equation}\label{eq:pi_f-defn}
    \pi_f(x,a) \coloneqq \#\{ p \leq x,~p\nmid N_f,~a_f(p)=a\}.
\end{equation}
One of the most celebrated such examples is the function
\begin{align}\label{eq: deltadefinition}
    \Delta(z) &= e^{2\pi iz}\prod_{n=1}^\infty (1 - e^{2\pi inz})^{24} = \sum_{n = 1}^\infty \tau(n) e^{2\pi i n z} \\
    &= e^{2\pi i z} - 24e^{4\pi i z} + 252e^{6\pi iz} - 1472e^{8\pi iz} + 4830e^{10\pi iz} - \cdots,
\end{align}
where $\tau(n)$ is Ramanujan's tau-function. Lehmer~\cite{Leh47} pondered whether $\tau(n) \ne 0$ for all $n \ge 1$, and showed that proving $\tau(p) \ne 0$ for all primes $p$ would imply this. The question remains open, with computations by Bosman \cite{EC11} verifying that $\tau(n) \ne 0$ for $n \le 2 \times 10^{19}$. 
Closely related to Lehmer's speculation is the following conjecture of Atkin and Serre~\cite{serre1974divisibilite}. 
\begin{conj*}[Atkin--Serre]
Let $f \in \mathcal{S}_k^{\mathrm{new}}(\Gamma_0(N))$ be a non-CM newform with trivial nebentypus. For all $\epsilon > 0$, there exist constants $c_f(\epsilon), c'_f(\epsilon) > 0$ such that if $p \ge c_f(\epsilon)$, then $|a_f(p)| \ge c'_f(\epsilon)p^{(k-3)/2 - \epsilon}$. 
\end{conj*}
In particular, the truth of this conjecture would imply for all $a \in \ZZ$ that $\tau(p) = a$ for only finitely many $p$. In light of this conjecture, it is natural to ask for bounds on the function $\pi_\Delta(x, a)$.
Serre \cite{Ser81} used refined versions of \eqref{eq: LO} and \eqref{eq: BT chebotarev bad} to show that for all $\delta < 1/4$, there exists an effectively computable constant $c(\delta) > 0$ such that
\begin{equation}\label{eq: Ser81}
    \pi_\Delta(x, a) \le c(\delta) \frac{x}{(\log x)^{1 + \delta}},\qquad x\geq 3.
\end{equation} 
Following subsequent improvements by Wan \cite{Wan} and V. K. Murty \cite{Mur97}, Thorner and Zaman~\cite{TZ} proved that there exists an absolute, effectively computable constant $c_9 > 0$ such that
\begin{equation}\label{eq: TZ pixa bound}
    \pi_\Delta(x, a) \leq c_9 \frac{x(\log \log x)^2}{(\log x)^2},\qquad x \ge 16.
\end{equation}

Under GRH for symmetric power $L$-functions, Rouse and Thorner \cite{RT} proved a stronger bound using arguments based on the Sato--Tate conjecture. They proved, under GRH for symmetric power $L$-functions, that there exists an absolute, effectively computable constant $c_{10} > 0$ such that for all $a \in \ZZ$, there exists an effectively computable constant $c(a) > 0$ such that
\begin{equation}
    \pi_\Delta(x, a) \le c_{10} \frac{x^{3/4}}{\sqrt{\log x}}, \qquad x \ge c(a).
\end{equation}
In the case $a = 0$, Zywina \cite{Zyw} matched this bound under GRH for Hecke $L$-functions using arguments based on the Chebotarev density theorem.

We make the result \eqref{eq: TZ pixa bound} explicit by proving the following.

\begin{thm}\label{thm: LT tau bound}
For all $a \in \ZZ$, we have that
\begin{align*}
    \pi_{\Delta}(x,a) \leq 4627 \frac{x (\log \log x)^2}{(\log x)^2}, \quad x \ge e^{e^{16}}.
\end{align*}
When $a = 0$, this bound may be strengthened to
\begin{align*}
    \pi_{\Delta}(x,0) \leq (3.01 \times 10^{-10}) \frac{x (\log \log x)^2}{(\log x)^2}, \quad x\ge e^{e^{16}}.
\end{align*}
\end{thm}

For the weight $2$ newform 
\begin{equation*}
f_E(z) = \sum_{n = 1}^\infty a_E(n) e^{2\pi inz} \in \mathcal{S}_2^{\text{new}}(\Gamma(N_E))
\end{equation*}
associated to a non-CM elliptic curve $E/\mathbb{Q}$ of conductor $N_E$, Lang and Trotter \cite{lang2006frobenius} conjectured that there exists a constant $c_{E, a} \geq 0$ such that

\begin{equation}
    \pi_E(x,a) \coloneqq \pi_{f_E}(x, a) \sim c_{E, a} \frac{\sqrt x}{\log x}.
\end{equation}
When $a = 0$, Elkies \cite{elkies1987existence, elkies1991distribution} and Fouvry and M. R. Murty \cite{fouvry1996distribution} showed that for all $\epsilon>0$, there exist positive constants $c(E,\epsilon)$, $c'(E,\epsilon)$, and $c(E)$ such that if $x\geq c(E,\epsilon)$, then
\begin{equation}
c'(E,\epsilon)\frac{\log\log\log x}{(\log\log\log\log x)^{1+\epsilon}}\leq \pi_E(x,a)\leq c(E)x^{3/4}.
\end{equation}
In particular, $a_E(p) = 0$ infinitely often. In the general case where $a \in \ZZ$ is fixed, Serre \cite{Ser81} proved that for all $\delta < 1/4$, there exists an effectively computable constant $c(N_E, \delta) > 0$ such that
\begin{equation}\label{eq: Serre supersingular bound}
    \pi_E(x, a) \le c(N_E, \delta) \frac{x}{(\log x)^{1 + \delta}},\qquad x\geq 3.
\end{equation}
After subsequent improvements by Wan \cite{Wan} and V. K. Murty \cite{Mur97}, Thorner and Zaman \cite{TZ} proved that there exists an effectively computable constant $c(N_E) > 0$ such that
\begin{equation}
    \pi_E(x, a) \le c(N_E) \frac{x(\log \log x)^2}{(\log x)^2}, \qquad x\geq 3.
\end{equation}
For an example, consider the elliptic curve
\begin{equation}\label{eq: elliptic curve}
    E: y^2 - y = x^3 - x^2.
\end{equation}
This elliptic curve has conductor $11$, the least conductor of any elliptic curve over $\mathbb{Q}$. Its associated modular form is given by
\begin{equation}\label{eq: E mod form}
    f_E(z) = \eta^2(z)\eta^2(11z), \qquad  \eta(z) \coloneqq e^{\pi iz/12} \prod_{n\ge 1} (1 - e^{2\pi inz}).
\end{equation}
 We prove the following theorem.

\begin{thm}\label{thm: level 2 bound}
Let $E$ be the elliptic curve given by \eqref{eq: elliptic curve} and $a \in \mathbb{Z}$. Then,
\begin{equation}
    \pi_E(x, a) \le 631 \frac{x(\log \log x)^2}{(\log x)^2}, \quad x \ge  e^{e^{13}}.
\end{equation}
\end{thm}

\begin{remark}
Analogues of Theorems \ref{thm: LT tau bound} and \ref{thm: level 2 bound} can be obtained for a broad class of other newforms, see \cref{rem: general lt}. 
\end{remark}
Suitably strong bounds for the quantity $\pi_f(x, 0)$ as in~\eqref{eq:pi_f-defn} can be used to study the proportion of integers $n \ge 1$ such that $a_f(n) \ne 0$. Define
\begin{equation*}
    D_f \coloneqq \lim_{x \to \infty} \frac{\#\{ 1 \leq n \leq x \mid a_f(n) \neq 0\}}{x}.
\end{equation*}
For $f = \Delta$, Serre~\cite[p.~379]{Ser81} proved that
\begin{equation}\label{eq: DDelta product}
    D_\Delta = \prod_{\tau(p) = 0} \Big( 1 - \frac{1}{p+1} \Big).
\end{equation}
As a consequence of \eqref{eq: Ser81}, this product converges absolutely. Rouse and Thorner~\cite{RT} showed that $D_{\Delta} > 1 - 1.54 \times 10^{-13}$ under the generalized Riemann hypothesis for the symmetric power $L$-functions of $\Delta$. Here, we prove the first unconditional positive lower bound for $D_\Delta$.
\begin{thm}\label{thm: Lehmer+}
We have
\begin{equation}
    D_\Delta > 1 - 1.15 \times 10^{-12}.
\end{equation}
\end{thm}

For the cusp form $f_E$ associated to the elliptic curve given in \eqref{eq: elliptic curve}, Serre \cite[p.~379]{Ser81} proved that

\begin{equation}\label{eq: Df product}
    D_{f_E} = \frac{14}{15}\prod_{a_E(p) = 0} \Big( 1 - \frac{1}{p+1} \Big).
\end{equation}
As a consequence of \eqref{eq: Serre supersingular bound}, the product converges absolutely. Serre showed that $D_{f_E} < 0.847$ and conjectured a lower bound of 0.845. Rouse and Thorner~\cite{RT} proved that $D_{f_E} > 0.8306$ under the generalized Riemann hypothesis for the symmetric power $L$-functions associated to $E$. We prove the first unconditional positive lower bound for $D_{f_E}$.
\begin{thm}\label{thm: level 2}
Let $E$ be given by \eqref{eq: elliptic curve}. We have that
\begin{equation*}
    D_{f_E} > 0.004769.
\end{equation*}
\end{thm}
The density $D_{\Delta}$ naturally arises in a central limit theorem for $\tau(n)$ proved by Luca, Radziwi\l\l, and Shparlinski~\cite{LRS}. Denote by $d(n)$ the number of divisors of $n$, and recall, by Deligne's proof of the Weil conjectures, that for all $n \geq 1$, the function $\tau(n)$ satisfies the bound
\begin{equation}\label{eq: Delignebound}
    |\tau(n)| \leq d(n)n^{11/2}.
\end{equation}
Consequently, it is natural to consider the distribution of a ``normalized'' tau-function $\tau(n)/n^{11/2}$ at the integers $n \geq 1$. For instance, 
the Sato--Tate conjecture implies that $|\tau(p)/p^{11/2}|$ stays arbitrarily near the value $2$ for a positive proportion of primes $p$. In fact, an application of Theorem 1 of \cite{LRS} demonstrates that for all $\epsilon > 0$, there exists a density one subset $S_{\epsilon}$ of the integers such that if $n\in S_{\epsilon}$, then
\begin{equation*}
    |\tau(n)|/n^{11/2} \leq (\log n)^{-\frac{1}{2} + \epsilon}.
\end{equation*}
This implies that for almost all $n \geq 1$, Deligne's bound \eqref{eq: Delignebound} for $\tau(n)/n^{11/2}$ is not sharp.

Moreover, Corollary 5 of \cite{LRS} shows that the exponent $-\frac{1}{2}$ above is optimal, in the sense that for fixed $v \in \RR$, the central limit theorem
\begin{equation*}
    \lim_{x \to \infty} \frac{ \#\Big\{1 \le n \le x ~\Big| ~\frac{\log |\tau(n)/n^{11/2}|+(\log \log n)/2}{\sqrt{(\frac{1}{2} + \frac{\pi^2}{12})\log \log n}}\geq v\Big\}}{\# \{ 1 \leq n \leq x \mid \tau(n) \neq 0\}} = \frac{1}{\sqrt{2\pi}} \int_v^\infty e^{-u^2/2}~du
\end{equation*}
holds.
This gives the distribution as a proportion of the integers in the support of $\tau(n)$, but by substituting the upper bound $D_{\Delta} \leq 1$ and our lower bound for $D_{\Delta}$ as in \cref{thm: Lehmer+}, we establish upper and lower limits for the distribution as a proportion of all integers in the style of the Erd\H{o}s--Kac theorem.
\begin{cor}
For fixed $v \in \mathbb{R}$, we have
\begin{align*}
    \frac{1}{\sqrt{2\pi}} \int_v^\infty e^{-u^2/2}~du &\geq \lim_{x \to \infty} \frac{1}{x} \cdot \#\bigg\{1 \le n \le x: ~\frac{\log |\tau(n)/n^{11/2}|+\frac{1}{2}\log \log n}{\sqrt{(\frac{1}{2} +\frac{\pi^2}{12})\log \log n}}\geq v\bigg\} \\&> (1-1.15 \times 10^{-12}) \frac{1}{\sqrt{2\pi}}\int_v^\infty e^{-u^2/2} ~du.
\end{align*}
\end{cor}
\subsection*{Outline of the paper}
In~\cref{sec:2}, we introduce notation and prove some preliminary technical lemmas, including estimates for a certain test function and bounds for Hecke $L$-functions in the critical strip. In \cref{sec: reduce}, we show that \cref{thm: BT general} follows from only considering abelian extensions of number fields. In~\cref{sec:3}, we use smoothing arguments to compute estimates for certain sums over integral ideals. We apply these estimates in~\cref{sec:4}, where we use the Selberg sieve over number fields to develop an upper bound for $\pi_C(x, L/K)$ for abelian extensions. In~\cref{sec: BT} we bound error terms arising from the Selberg sieve by choosing $x$ in a large enough range relative to parameters depending on $L/K$, thereby yielding our Brun--Titchmarsh theorem for abelian extensions.

In~\cref{sec: LT}, we deduce an upper bound for $\pi_f(x,a)$ from upper bounds for $\pi_C(x, L/K)$ for suitably chosen extensions $L/K$ arising from $\ell$-adic Galois representations attached to $f$. In~\cref{sec: nonzero} we make the results in \cref{sec: LT} explicit for $\tau(n)$, proving \cref{thm: LT tau bound} and \cref{thm: Lehmer+}. Finally, in \cref{sec: elliptic curves} we prove Theorems \ref{thm: level 2 bound} and \ref{thm: level 2} for the elliptic curve \eqref{eq: elliptic curve}.
\subsection*{Acknowledgments}
The authors would like to thank 
Ken Ono for his valuable suggestions. They are grateful for the support of grants from the National Science Foundation
(DMS-2002265, DMS-2055118, DMS-2147273), the National Security Agency (H98230-22-1-0020), and the Templeton World Charity Foundation. This research was conducted as part of the 2022 Research Experiences for Undergraduates at the University of Virginia.

%% file: 2-notation.tex
Throughout this paper $s = \sigma + it$ is a complex variable.

\subsection{Notation}
In this section we establish the notation utilized in the proof of Theorem \ref{thm: BT general}. Subsequently $K$ will be a number field and $L$ a finite extension of $K$. We use the following notation:
\begin{itemize}
    \item $\mathcal{O}_K$ is the ring of integers of $K$.
    \item $n_K = [K:\Q]$ is the degree of $K/\Q$.
    \item $D_K = |\disc(K/\Q)|$.
    \item $ \N_{K/\Q}$ is the absolute norm of $K$. We write $\N = \N_{K/\mathbb{Q}}$ when there is no ambiguity about the field.
    \item $\zeta_K(s)$ is the Dedekind zeta function of $K$.
    \item $\kappa_K$ is the value of the residue of the pole of $\zeta_K(s)$ at $s = 1$.
    \item $h_K$ is the class number of $K$.
    \item $\mathfrak{p}$ denotes a prime ideal of $K$.
    \item $\mathfrak{a}$ denotes an integral ideal of $K$.
    \item $\omega(\mathfrak{a})$ is the number of distinct prime ideals dividing $\mathfrak{a}$.
    \item $\varphi_K(\mathfrak{a}) = (\N\mathfrak{a}) \prod_{\mathfrak{p} \mid \mathfrak{a}} ( 1 - \N\mathfrak{p}\inv)$ is the Euler phi-function associated to $K$.
\end{itemize}

Suppose that $L/K$ is a Galois extension of number fields, and let $G$ be its Galois group. Recall the \emph{Artin symbol} $( \frac{L/K}{\mathfrak{p}})$, also denoted $\Frob_{\mathfrak{p}}$, associated to a prime ideal $\mathfrak{p}$ of $K$ unramified in $L$; it is the conjugacy class in $G$ consisting of Frobenius automorphisms corresponding to the prime ideals $\mathfrak{P}$ lying over $\mathfrak{p}$.

Let $L/K$ be abelian. Then, each conjugacy class of $G$ is a singleton. Thus, the Artin symbol maps unramified prime ideals of $K$ to elements of $G$. For $\mathfrak{m}$ an integral ideal of $K$, let $I(\mathfrak{m})$ denote the group of fractional ideals relatively prime to $\mathfrak{m}$, and let $P_{\mathfrak{m}}$ be the subgroup of principal ideals $( \alpha )$ with $\alpha \in K^\times$ totally positive and $\alpha \equiv 1 \Mod{\mathfrak{m}}$. Then the \emph{ray class group} modulo $\mathfrak{m}$ is defined to be the finite group $I(\mathfrak{m})/P_{\mathfrak{m}}$. 

If $\mathfrak{m}$ is divisible by all ramified primes in $L/K$, then the Artin symbol extends multiplicatively to a surjective group homomorphism called the \emph{Artin map}
\begin{equation*}
    F_{L/K}: I(\mathfrak{m}) \to \Gal(L/K).
\end{equation*}
By Artin reciprocity, there exists an integral ideal $\mathfrak{m}$ such that
\begin{equation*}
    H = \ker(F_{L/K})
\end{equation*}
contains $P_{\mathfrak{m}}$. In this case, $F_{L/K}$ induces an isomorphism
\begin{equation}\label{eq:artin-reciprocity}
    I(\mathfrak{m})/H \cong \Gal(L/K).
\end{equation}
Moreover, there exists a unique integral ideal $\mathfrak{f}_{L/K}$, called the \emph{Artin conductor} of $L/K$, dividing all $\mathfrak{m}$ with this property. Thus $[I(\mathfrak{f}_{L/K}) : H] = [L:K]$. 

Let $C$ be one conjugacy class of $G$. While our goal is to bound the quantity $\pi_C(x,L/K)$, the above discussion implies that this problem is equivalent to the problem of counting prime ideals in a given coset of $H$ in the ray class group modulo $\mathfrak{f}_{L/K}$. In particular, our conjugacy class $C \subset G$ is identified with a unique coset $aH \in I(\mathfrak{f}_{L/K})/H$, so the prime ideals $\mathfrak{p}$ with $\Frob_{\mathfrak{p}} = C$ are precisely the prime ideals $\mathfrak{p} \in aH$. In summary, we have that
\begin{equation}\label{eq: translation to cosets}
    \pi_C(x,L/K) = \pi_{aH}(x,L/K),
\end{equation}
where $\pi_C(x, L/K)$ is given by \eqref{eq: piC def} and
\begin{equation}
    \pi_{aH}(x,L/K) \coloneqq \#\Big\{ \mathfrak{p} \text{ unramified in } L/K,~\mathfrak{p} \in aH,~\N_{K/\mathbb{Q}}(\mathfrak{p}) \leq x \Big\}.
\end{equation}

We recall the basic theory of ray class characters and their associated $L$-functions. Let $\mathfrak{m}$ be an integral ideal. A \emph{ray class character} $\chi$ modulo $\mathfrak{m}$ is defined to be a character of the group $I(\mathfrak{m})/P_{\mathfrak{m}}$, with the convention that $\chi(\mathfrak{a}) = 0$ if $\mathfrak{a}$ and $\mathfrak{m}$ are not coprime. Note that we will often view $\chi$ as acting on an ideal $\mathfrak{a}$ itself rather than its ideal class $\mathfrak{a}P_\mathfrak{m}$ in the ray class group. If $\mathfrak{m}$ divides $\mathfrak{n}$ then the inclusion induces a map $I(\mathfrak{n})/P_{\mathfrak{n}} \to  I(\mathfrak{m})/P_{\mathfrak{m}}$; composing a character $\chi$ mod $\mathfrak{m}$ with this map induces a character $\chi'$ mod $\mathfrak{n}$. 
A character is \emph{primitive} if it cannot be induced, except by itself. Given some $\chi$ mod $\mathfrak{m}$, there is a unique primitive character $\widetilde{\chi}$ mod $\mathfrak{f}_{\chi}$ which induces $\chi$. The integral ideal $\mathfrak{f}_{\chi}$ is the \emph{conductor} of $\chi$.

If $\mathfrak{m} = \mathfrak{f}_{L/K}$ and $H = \ker(F_{L/K})$, then we define the following quantity associated to the extension $L/K$,
\begin{equation*}
    \mathcal{Q} = \mathcal{Q}(L/K) \coloneqq \max_{\chi(H) = 1} \N_{K/\mathbb{Q}}\mathfrak{f}_{\chi},
\end{equation*}
where the maximum is over characters $\chi$ such that $\chi(H) = 1$; by~\eqref{eq:artin-reciprocity}, these correspond precisely to the irreducible characters of $G$.


Given an integral ideal $\mathfrak{m}$ and a ray class character $\chi$ modulo $\mathfrak{m}$, the \emph{Hecke $L$-function} associated to $\chi$ is given by the following Dirichlet series and Euler product:
\begin{equation*}
    L(s,\chi) = \sum_{\mathfrak{a}} \frac{\chi(\mathfrak{a})}{\N\mathfrak{a}^s} = \prod_{\mathfrak{p}} \Big( 1 - \frac{\chi(\mathfrak{p})}{\N\mathfrak{p}^s} \Big)^{-1}, \qquad \sigma > 1,
\end{equation*}
with the sum over all integral ideals of $K$ and the product over all prime ideals of $K$.

For the purpose of obtaining explicit numerical constants, careful bounds on Hecke $L$-functions along vertical lines near the critical strip will also be of importance. These are derived by application of the Phragm\'{e}n--Lindel\"{o}f principle for complex analytic functions in Lemma \ref{lem: Lfunctionbound}, taking as input the functional equation for the the Hecke $L$-functions and asymptotics for related gamma factors. First we define the quantities
\begin{align*}
    D_{\chi} &= D_K \N_{K/\mathbb{Q}}\mathfrak{f}_{\chi}. \\
    \delta_{\chi_0}(\chi) &= \begin{cases} 1 & \chi = \chi_0 \\
    0 & \textrm{otherwise} 
    \end{cases}
\end{align*}
Let $r_1$ be the number of real places of $K$ and $2r_2$ the number of complex places of $K$. If $\chi$ is primitive, then let $0 \leq \mu_{\chi} \leq r_1$ be the number of real places at which $\chi$ ramifies, and set
\begin{align*}
    A_{\chi} &= 2^{-r_2}\pi^{n_K/2}D_{\chi}^{1/2}, \\
    \Gamma_{\chi}(s) &= \Gamma\Big( \frac{s}{2} \Big)^{r_1 - \mu_{\chi}} \Gamma\Big( \frac{s+1}{2} \Big)^{\mu_{\chi}} \Gamma(s)^{r_2}.
\end{align*}
The completed Hecke $L$-function is given by
\begin{equation*}
    \Lambda(s,\chi) = A_{\chi}^s \Gamma_{\chi}(s) L(s,\chi).
\end{equation*}
For $\chi$ a primitive character, $\Lambda(s, \chi)$ satisfies the  functional equation
\begin{equation*}
    \Lambda(1 - s, \overline{\chi}) = W(\chi) \Lambda(s,\chi),
\end{equation*}
with the root number $W(\chi)$ a complex constant of modulus $1$.

\subsection{Preliminary estimates}
We begin by establishing explicit bounds for the Hecke $L$-function of a ray class character in the critical strip. 
\begin{lem}\label{lem: Lfunctionbound}
Let $\mathfrak{m}$ be an integral ideal of $K$, and let $\chi$ be a primitive ray class character modulo $\mathfrak{m}$. If $\delta > 0$ and $-\delta < \sigma < 1$, then
\begin{equation}
    |L(s, \chi)| \le \Big( \frac{\delta + 1}{1 - \sigma} \Big)^{\delta_{\chi_0}(\chi)}(1 + \delta\inv)^{n_K} \Big( \frac{D_\chi}{(2\pi)^{n_K}}|s+1|^{n_K} \Big)^{(1+ \delta - \sigma)/2}.
\end{equation}
\end{lem}
\begin{proof}
Expanding the functional equation into unsymmetric form and taking absolute values yields
\begin{equation*}
    |L(s,\chi)| = |L(1 - s,\chi)| \Big| \frac{\Gamma_{\chi}(1 - s)}{\Gamma_{\chi}(s)} \Big| A_{\chi}^{1 - 2s}.
\end{equation*}
Using \cite[Lemmas 1, 2, 3]{Rad}, we bound the gamma factors at $ s= -\delta + it$ with $\delta \in (0, 1/2)$:
\begin{align*}
    \frac{\Gamma_\chi(1-s)}{\Gamma_\chi(s)} &= \Big(\frac{ \Gamma( \frac{1-s}{2})}{ \Gamma(\frac{s}{2})} \Big)^{r_1 - \mu_\chi} \Big(\frac{ \Gamma( \frac{1}{2} + \frac{1 - s}{2})}{ \Gamma(\frac{1}{2} + \frac{s}{2})} \Big)^{ \mu_\chi} \Big(\frac{ \Gamma( 1-s)}{ \Gamma(s)} \Big)^{r_2} \\
    &\le \Big( \frac{1}{2} | s + 1 | \Big)^{(1/2 + \delta) (r_1 - \mu_\chi)} \Big( \frac{1}{2} | s + 1/2 | \Big)^{ (1/2 + \delta) \mu_\chi} \Big(  | s + 1 | \Big)^{(1 + 2\delta) r_2} \\
    &\le 2^{-(1/2 + \delta) r_1} |s+1|^{(1/2 + \delta)n_K}
\end{align*}
where we use that $r_1 + 2r_2 = n_K$. Therefore, if $\delta \in (0,  1/2)$, then
\begin{align*}
    |L(-\delta + it, \chi)| \le |L(1 + \delta + it, \chi)| \Big(\frac{D_\chi}{(2 \pi)^{n_K}} |s+1|^{n_K} \Big)^{ (1/2 + \delta) }.
\end{align*}
By the estimate 
\begin{equation*}
    |L(s,\chi)| \leq \zeta_K(\sigma) \leq \zeta(\sigma)^{n_K} \leq \Big( 1 + \frac{1}{\sigma - 1} \Big)^{n_K}, \quad \sigma > 1,
\end{equation*}
we also have the bound
\begin{align*}
    |L(1+\delta + it, \chi)| \le (1+\delta\inv)^{n_K}.
\end{align*}
Now, assume that $\chi$ is nontrivial, so $L(s,\chi)$ is holomorphic in the strip $-\delta \le \sigma \le 1 + \delta$. Then in the strip $-\delta \le \sigma \le 1 + \delta$, applying the Phragm\'{e}n--Lindel\"{o}f principle gives
\begin{align*}
    |L(s, \chi)| \le (1 + \delta\inv)^{n_K}  \Big(\frac{D_\chi}{(2 \pi)^{n_K}} |s+1|^{n_K} \Big)^{ (1 + \delta - \sigma)/2 }.
\end{align*}

When $\chi$ is trivial, we have that $L(s, \chi) = \zeta_K(s)$. Since $\zeta_K(s)$ has a pole at $s = 1$, we choose any $B \geq 1$ and we instead estimate $\frac{s+B}{s-1}\zeta_K(s)$. By a similar argument as above, in the strip $- \delta \leq \sigma \leq \delta +1 $ we have that 
\begin{equation}
    |\zeta_K(\sigma + it)| \le \Big| \frac{\delta + 1}{\delta - B} \Big| \Big| \frac{s+B}{s-1} \Big| (1 + \delta\inv)^{n_K}  \Big(\frac{D_K}{(2 \pi)^n} |s+1|^{n_K} \Big)^{ (1 + \delta - \sigma)/2 }.
\end{equation}
Using the bound
\begin{equation}
    \Big| \frac{s +B}{s -1} \Big| = \Big| \frac{B+1}{s -1} + 1 \Big| \le \frac{B+1}{1- \sigma} + 1 = \frac{B+\sigma}{1- \sigma},
\end{equation}
we find that 
\begin{equation}\label{eq: zeta PL}
    |\zeta_K(\sigma + it)| \leq  \Big(\frac{\delta+1}{1- \sigma} \Big)(1 + \delta \inv)^{n_K}   \Big(\frac{D_K}{ (2\pi)^{n_K}} |s +1|^{n_K} \Big)^{(1 + \delta - \sigma)/2},
\end{equation}
since $|B+\sigma | / |\delta - B|$ can be made arbitrarily close to 1 as $B \to \infty$.
\end{proof}

In order to obtain explicit estimates for $\pi_C(x, L/K)$, we utilize a test function which approximates the indicator function of the interval $[\frac{1}{2}, 1]$. 
\begin{lem}[Thorner--Zaman~\protect{\cite[Lemma~2.2]{TZ}}]\label{lem:test-fn}
For any $x \geq 3$, $\epsilon \in (0,\frac{1}{2})$, and positive integer $\ell \geq 1$, set
\begin{equation*}
    A = \frac{\epsilon}{2\ell \log x}.
\end{equation*}
There exists a real-variable function $\phi(t) = \phi(t;x,\ell,\epsilon)$ such that:
\begin{enumerate}
    \item $0 \leq \phi(t) \leq 1$ for all $t \in \mathbb{R}$, and $\phi(t) \equiv 1$ for $\frac{1}{2} \leq t \leq 1$.
    \item The support of $\phi$ is contained in the interval $[ \frac{1}{2} - \frac{\epsilon}{\log x}, 1 + \frac{\epsilon}{\log x}]$.
    \item Its Laplace transform $F(z) = \int_{\mathbb{R}} \phi(t)e^{-zt} \, dt$ is entire and is given by
    \begin{equation}\label{eq: test function Laplace}
        F(z) = e^{-(1 + 2\ell A)z} \cdot \Big( \frac{1 - e^{(\frac{1}{2} + 2\ell A)z}}{-z} \Big)\Big( \frac{1 - e^{2Az}}{-2Az} \Big)^{\ell}.
    \end{equation}
    \item Let $s = \sigma + it \in \mathbb{C}$, $\sigma > 0$ and let $\alpha$ be any real number satisfying $0 \leq \alpha \leq \ell$. Then
    \begin{equation}
        |F(-s \log x)| \leq \frac{e^{\sigma \epsilon}x^{\sigma}}{|s| \log x} \cdot (1 + x^{-\sigma/2}) \cdot \Big( \frac{2\ell}{\epsilon|s|} \Big)^{\alpha}.
    \end{equation}
    \item We have
    \begin{equation}
         F(0) = \frac{1}{2} + \frac{\epsilon}{\log x} \quad \textrm{and} \quad F(- \log x) \leq e^{\epsilon} \frac{ x}{ \log x}.
    \end{equation}
    
    \item If $ s = \sigma + it \in \CC$ and $\sigma \le 0 $, then
    \begin{equation}
        |F(-s \log x)| \le \frac{2x^{\sigma/2}}{|s| \log x} \cdot \Big( \frac{ \ell(1 + e^{-\epsilon \sigma})}{\epsilon |s|} \Big)^\ell.
    \end{equation}
    
    \item If $s = \sigma + it$ with $\sigma < 0$ and $|s| < \frac{2\ell}{\epsilon}$, we have that 
    \begin{equation}
        |F(-s \log x)| \le \frac{2x^{\sigma/2}}{|s| \log x} \cdot \Big(\frac{2}{2 - |s\epsilon/\ell|} \Big)^\ell.
    \end{equation}
    
\end{enumerate}
\end{lem}

\begin{proof}
Parts (a)--(e) are proven in Lemma 2.2 of \cite{TZ}. Part (f) follows from \eqref{eq: test function Laplace}. Part (g) follows from \eqref{eq: test function Laplace} and the following identity, which is valid for any complex $|z| < 2$:
\begin{equation}
    \Big|\frac{1 - e^z}{-z}\Big| \le \frac{2}{2 - |z|}.
\end{equation}
Thorner and Zaman \cite{TZ} 
additionally assumed that $\epsilon \in (0, \frac{1}{4})$ in order to obtain bounds for $F(-s\log x)$ on the line $\sigma = -1/2$. We have no need for such estimates, so we can relax this to $\epsilon \in (0, \frac{1}{2})$.
\end{proof}

\subsection{Auxiliary results}
Assume throughout that $L/K$ is an abelian extension of number fields. The following lemmas will also be useful.
\begin{lem}\label{lem: weakidealbound}
For all $0 < \gamma \le 1$, we have
\begin{equation}
    \sum_{\N(\a) \le x} 1 \le (1 + \gamma\inv)^{n_K} x^{1 + \gamma}.
\end{equation}
\end{lem}
\begin{proof}
See~\cite[Lemma~1.12(a)]{Wei}.
\end{proof}

\begin{lem}\label{abc}
Let $a, b > 0$. If $x \ge (ab)^{a(\sqrt{2/\log(ab)} + 1)}$, then
\begin{equation}\label{eq: inequality}
    \frac{x^{1/a}}{\log x} \ge b.
\end{equation}
\end{lem}
\begin{proof}
    By properties of the Lambert $W$-function \cite{corless1996lambertw}, we have that \eqref{eq: inequality} holds when
    \begin{equation*}
        \log x \ge -a W_{-1} \Big( \frac{-1}{a b} \Big)
    \end{equation*}
    where $W_{-1}(z)$ is the negative branch of the Lambert $W$-function. By Theorem 1 of \cite{chatzigeorgiou2013bounds}, we have that
    \begin{equation*}
        -a W_{-1} \Big( \frac{-1}{a b} \Big) \le a \Big(\sqrt{2\log (ab)} + \log(ab) \Big) \le a (\sqrt{2/\log(ab)} + 1) \log (ab),
    \end{equation*}
    and the lemma follows by taking exponents.
\end{proof}

Next, we bound $[L:K]$ in terms of $n_K$, $D_K$ and $\mathcal{Q}$ alone.

\begin{lem}\label{lem: degree bound}
For all $0 < \epsilon \leq 1$,
\begin{equation*}
    [L:K] \leq 2^{n_K}(1 + \epsilon^{-1})(n_K e)^{1 + \epsilon}\Big(\frac{4}{e^2\pi}\Big)^{n_K(1 + \epsilon)/2}D_K^{(1 + \epsilon)/2}\mathcal{Q}.
\end{equation*}
At $\epsilon = 1$, for instance, this yields
\begin{equation*}
    [L:K] \leq (14.779)(0.7169)^{n_K} D_K \mathcal{Q}.
\end{equation*}
\end{lem}
\begin{proof}
By \cite[Lemma 1.16]{Wei}, for $H$ the congruence subgroup attached to the Artin conductor $\mathfrak{f}_{L/K}$, we have $[L:K] = [I(\mathfrak{f}_{L/K}):H] \leq 2^{n_K} h_K \mathcal{Q}$. By Minkowski's bound and \cite[Lemma 1.12]{Wei},
\begin{equation*}
    h_K \leq \sum_{\N(\mathfrak{a})\leq x_0}1 \leq (1 + \epsilon^{-1})^{n_K} x_0^{1 + \epsilon}
\end{equation*}
for $x_0 = \frac{n!}{n^n}(\frac{4}{\pi})^{n/2}D_K^{1/2} \leq (\frac{n}{e^{n-1}})(\frac{4}{\pi})^{n/2}D_K^{1/2}$. Then, an explicit bound is
\begin{equation*}
    h_K \leq (1 + \epsilon^{-1})(n_K e)^{1 + \epsilon}\Big(\frac{4}{e^2\pi}\Big)^{n_K(1 + \epsilon)/2}D_K^{(1 + \epsilon)/2}.
\end{equation*}
Note that for $n_K \ge 1$ we have $n_K \le e^{n_K -1}$. At $\epsilon = 1$, note that $n^2 \leq (3^{2/3})^n$ for all $n \in \mathbb{Z}_{\geq 1}$.
\end{proof}
Next, we set $\mathfrak{m} = \mathfrak{f}_{L/K}$ and state the following bound for the quantity $\omega(\m)$.

\begin{lem}\label{lem: divisor bound}
Let $\mathfrak{m} = \mathfrak{f}_{L/K}$. For every $b > 0$ we have that
\begin{equation}
    \omega(\m) < 2e^{1 + 2/b}n_K + b \log D_K \mathcal{Q}.
\end{equation}
\end{lem}
\begin{proof}
See~\cite[Lemma~1.13(b)]{Wei}.
\end{proof}

Next, we introduce and bound the following product factor to appear in our treatment of non-primitive Hecke $L$-functions.

\begin{lem}\label{lem: Pmd bound}
Let $M, b, \sigma, x > 0$. Define
\begin{equation}\label{eq: Pdef}
    Z_\m(\sigma) \coloneqq \prod_{\mathfrak{p} | \m} (1 +  \N\mathfrak{p}^{-\sigma}).
\end{equation}
Let $\pi(x) \coloneqq \#\{p \leq x,~ p \text{ prime}\}$. We have that
\begin{align}
    \log Z_\m(\sigma) &\le n_K \Big[\log ( 1 + M^{-\sigma} )(2e^{1 + 2/b} - \pi(M-1)) + \sum_{p < M} \log (1 + p^{-\sigma}  ) \Big]\nn \\
    &\phantom{\,\,\,\,\,\,\,\,}+ b\log ( 1 + M^{-\sigma}) \log (D_K \mathcal{Q}).
\end{align}
\end{lem}
\begin{proof}
Given a rational prime $p$, there are at most $n_K$ prime ideals in $K$ with norm $p$. Thus, the product over all the factors in $Z_\m(\sigma)$ with norm at most $M$ is bounded by 
\begin{equation}\label{eq: Pm small p}
    \prod_{p < M}(1 + p^{-\sigma})^{n_K}.
\end{equation}
There are at most $\omega(\m) - n_K \pi(M-1)$ prime ideals $\mathfrak p$ dividing $\m$ such that $\N(\mathfrak{p}) \ge M$. In the product for $Z_\m(\sigma)$, the factor corresponding to these primes can be bounded above by $(1 + M^{-\sigma})$, so that the product over all the factors in $Z_\m(\sigma)$ with norm at least $M$ is bounded by 
\begin{equation}\label{eq: Pm big p}
    (1 + M^{-\sigma})^{\omega(\m) - n_K \pi(M-1)}.
\end{equation}
Multiplying equations \eqref{eq: Pm small p} and \eqref{eq: Pm big p}, taking logarithms, and applying Lemma \ref{lem: divisor bound} completes the proof.
\end{proof}

Note that we could trivially bound $Z_\mathfrak{m}(\sigma)$ by $2^{\omega(\mathfrak{m})}$, but using \cref{lem: Pmd bound} allows us to separately bound the contribution from small divisors of $\mathfrak{m}$ by choosing $M$ optimally, which provides sharper bounds later on.

%% file: 2.5-reduceToAbelian.tex
\section{Reduction to the abelian case}\label{sec: reduce}

In this section, we show that \cref{thm: BT general} follows from the following version of the Brun--Titchmarsh theorem for abelian extensions, which we prove in \cref{sec: BT}. This allows us to specialize to abelian extensions in Sections \ref{sec:3} and \ref{sec:4}, wherein we develop auxiliary technical results.

\begin{lem}\label{lem: BT abelian}
Let $L/K$ be an abelian extension of number fields, let $C$ be a conjugacy class (which will necessarily be a singleton) of $\Gal(L/K)$, and let $\pi_C(x,L/K)$ be the prime-counting function given in \eqref{eq: piC def}. If
\begin{equation}\label{eq: abelian range}
    x \ge e^{36}e^{92n}(D_K\mathcal{Q})^{8.4} n^{4.2n},
\end{equation}
then
\begin{equation}
    \pi_C(x, L/K) \le 11.29 \frac{x}{[L:K] \log x}.
\end{equation}
\end{lem}

We show that \cref{thm: BT general} follows from this theorem for a general (nonabelian) Galois extension of number fields $L/F$. To do this, we change the base field from $F$ to an intermediate field $K$ such that $L/K$ is abelian, so that we can apply \cref{lem: BT abelian}.

The key tool is the following result from \cite[Proof of Proposition 3.9]{MMS}.
\begin{lem}[Murty--Murty--Saradha~\protect{\cite{MMS}}]\label{lem: base change}
Let $L/F$ be a Galois extension of number fields with Galois group $G$, and let $C \subset G$ be a conjugacy class. Let $A$ be an abelian subgroup of $G$ such that $C \cap A$ is nonempty, and let $K$ be the fixed field of $A$. Let $g \in C \cap A$, and let $C_A = C_A(g)$ denote the conjugacy class of $A$ which contains $g$. If $x \geq 2$, then
\begin{equation*}
    \left|\pi_C(x, L/F)-\frac{|C|}{|G|}\frac{|A|}{|C_A|}\pi_{C_A}(x, L/K)\right| \leq \frac{|C|}{|G|}\left(n_L x^{1/2} + \frac{2}{\log 2}\log D_L\right).
\end{equation*}

\end{lem}
\begin{proof}[Proof of \cref{thm: BT general}]
Let $L/F$ be a Galois extension of number fields and let $C$ be a conjugacy class of the Galois group $G = \Gal(L/F)$. Let $A$ be an abelian subgroup of $G$ such that $C \cap A$ is nonempty, and let $K$ be the fixed field of $A$. We apply \cref{lem: BT abelian} to the abelian extension $L/K$. If \eqref{eq: abelian range} holds, then
\begin{equation}
    n_L x^{1/2} = n_K [L:K]x^{1/2} \le \frac{x}{200\log x}.
\end{equation}
By the conductor discriminant formula we have that $D_L \le D_K \mathcal{Q}^{ [L : K]}$. Thus if \eqref{eq: abelian range} holds, then
\begin{equation}
    \frac{2}{\log 2} \log D_L \le \frac{2}{\log 2} (\log D_K  + [L:K]\log \mathcal{Q}) \le \frac{x}{200 \log x}.
\end{equation}
Thus by Lemmas \ref{lem: BT abelian} and \ref{lem: base change} we have that if \eqref{eq: abelian range} holds, then
\begin{equation}
    \pi_C(x, L/F) \le 11.3 \frac{|C|}{|G|} \frac{x}{\log x}
\end{equation}
completing the proof of \cref{thm: BT general}.
\end{proof}

To prove \cref{lem: BT abelian}, we utilize a version of the Selberg sieve developed by Weiss \cite{Wei}. First, we need explicit bounds on certain sums over ideals, which we develop in \cref{sec:3}. We use these bounds in \cref{sec:4}, where we implement the Selberg sieve.

%% file: 3-aux_estimates.tex
Let $aH$ be a coset of $I(\mathfrak{m})/H$ and $\mathfrak{n}$ an integral ideal coprime to $\mathfrak{f}_{L/K}$. Define
\begin{equation}
    A(x;a,\mathfrak{n}) \coloneqq \sum_{\substack{\mathfrak{a} \in aH \\ \mathfrak{n} \mid \mathfrak{a}}} \phi\Big( \frac{\log \N(\mathfrak{a})}{\log x} \Big),
\end{equation}
where $\phi$ is the test function described in~\cref{lem:test-fn}.
We obtain estimates for $A(x; a, \n)$ in Section \ref{sec: Axan}.

The second quantity we consider is
\begin{equation}\label{eq: Vz def}
    V(z) \coloneqq \sum_{\N(\mathfrak{a}) \leq z} \frac{1}{\N(\mathfrak{a})}.
\end{equation}
We establish lower bounds for $V(z)$ in Section \ref{sec: Vz}.



\subsection{Bounding $A(x;a,\mathfrak{n})$} \label{sec: Axan}

We begin by computing a smoothed character sum over integral ideals.

\begin{lem}\label{lem: character sum}
Let $x \ge 3$ and $\epsilon, \delta \in (0, 1/2)$. Let $\chi$ be a ray class character of $K$ modulo $\mathfrak{m}$, and let $\mathfrak{n}$ be an integral ideal coprime to $\mathfrak{m}$. Let $Z_\mathfrak{m}(\delta)$ be defined as in \eqref{eq: Pdef}. If
\begin{equation}\label{eq: E def}
    E(x) \defeq Z_\m(\delta) \frac{8}{3\pi } \frac{1 + \delta}{\delta( 1-  \delta)}( 1 + \delta\inv)^{n_K} e^{\delta \epsilon} \Big( \frac{D_K \mathcal{Q}}{ (2\pi)^{n_K}} \Big)^{1/2}  2^{n_K/2}  \Big[\frac{2 }{\epsilon } (1 + \epsilon) \Big]^{n_K/2 + 1} n_K^{n_K/2} x^\delta,
\end{equation}
then
\begin{equation*}
    \Big|\sum_{\substack{\mathfrak{a} \\ \n | \a }} \chi(\a) \phi \Big (\frac{ \log \N\a }{\log x} \Big) - \delta_{\chi_0}(\chi) \chi(\mathfrak{n}) \frac{\varphi_K(\mathfrak{m})}{\N\mathfrak{m}} \frac{\kappa_K}{\N\n} \cdot \log x \cdot F(-\log x) \Big| \leq E(x),
\end{equation*}
\end{lem}
\begin{proof}
By Laplace inversion, for $\delta > 0$ we have
\begin{equation*}
    \sum_{\substack{\mathfrak{a} \\ \n | \a }} \chi(\a) \phi\Big(\frac{\log \N\mathfrak{a}}{\log x}\Big) = \chi(\mathfrak{n})\frac{\log x}{2\pi i}\int_{1 + \delta -i\infty}^{1 + \delta +i\infty} \frac{L(s, \chi)}{(\N\n)^s} F(-s\log x)\,ds.
\end{equation*}
We shift the line of integration from $\Re(s) = 1+\delta$ to $\Re(s) = \delta$, justifying the vanishing of the horizontal integrals by rapid decay of the function $F(-s \log x)$ as $|t| \to \infty$. In doing so, we pick up a residue at $s = 1$, so the above is equal to
\begin{equation*}
    \delta_{\chi_0}(\chi) \chi(\mathfrak{n}) \frac{\varphi_K(\mathfrak{m})}{\N\mathfrak{m}} \frac{\kappa_K}{\N\n}  \log x \cdot F(-\log x)  + \chi(\mathfrak{n})\frac{\log x}{2\pi i}\int_{\delta-i\infty}^{\delta+i\infty} \frac{L(s, \chi)}{(\N\n)^s} F(-s\log x)\,ds.
\end{equation*}

To estimate the integral, we first express $L(s, \chi)$ in terms of an $L$-function of a primitive character. If $\widetilde{\chi}$ is the primitive character which induces $\chi$, we have that 
\begin{equation}\label{eq: primitive L}
    L(s, \chi) = L(s, \widetilde{\chi}) \prod_{\mathfrak{p} | \m} (1 - \widetilde{\chi} (\mathfrak{p})\N\mathfrak{p}^{-s}).
\end{equation}
We can bound the product over primes in \eqref{eq: primitive L} by $Z_\m(\sigma)$, defined in \eqref{eq: Pdef}. Thus we have
\begin{equation}\label{eq: make integral primitive}
    \Big|\chi(\n) \frac{\log x }{2\pi i}\int_{\delta-i\infty}^{\delta+i\infty}\frac{L(s, \chi)}{(\N\n)^s}F(-s \log x) \,ds\Big| \le \frac{\log x}{2\pi} Z_\m(\delta) \int_{\delta-i\infty}^{\delta+i\infty} |L(s, \widetilde{\chi})| |F(-s \log x)| \,ds.
\end{equation}
We split the remaining integral into two pieces, the first with $|t| \le M$ and the second with $|t| > M$, for some constant $M > 0$ which will be specified later. 
We bound the contribution from $L(s, \chi)$ with \cref{lem: Lfunctionbound} and the contribution from $F(-s\log x)$ with \cref{lem:test-fn}(d). Using these bounds and \eqref{eq: make integral primitive} gives
\begin{align*}
     &\Big|\chi(\n) \frac{\log x }{2\pi i}\int_{\delta-i\infty}^{\delta+i\infty}L(s, \chi) F(-s \log x) \,ds\Big| \nn \\
    &\le  Z_\m(\delta) \frac{1}{\pi} \Big( \frac{D_\chi}{ (2\pi)^{n_K}} \Big)^{1/2} \Big(\frac{1 + \delta}{ 1-  \delta} \Big)( 1 + \delta\inv)^{n_K} e^{\delta \epsilon} x^\delta ( 1+ x^{-\delta/2}) \nonumber \\
    &\times \Big[ \int_{0}^M  \frac{|\delta + 1 + it |^{n_K/2}}{|\delta + it|  } \,dt +  \Big( \frac{2 \ell}{\epsilon} \Big)^{n_K/2 + 1} \int_{M}^\infty  \frac{|\delta + 1 + it |^{n_K/2}}{|\delta + it|^{n_K/2 + 2} } \,dt \Big].
\end{align*}
We bound the first integral by
\begin{align*}
     \int_{0}^M  \frac{|\delta + 1 + it |^{n_K/2}}{|\delta + it|  }  \,dt &\le \frac{1}{\delta} \int_{0}^M  (\delta + 1 + t )^{n_K/2} \,dt\nn \\
    &\le \frac{1}{\delta(n_K/2 + 1)} (\delta + 1 + M)^{n_K/2 + 1}.
\end{align*}
We set $\ell = \lceil n_K/2 + 1 \rceil$ and substitute $u = 1 + t\inv$ to bound the second integral by
\begin{align}
    \int_{M}^\infty  \frac{|\delta + 1 + it |^{n_K/2}}{|\delta + it|^{\alpha + 1} } \,dt  &\le \int_{M}^\infty  \Big| 1 + \frac{1}{\delta + it} \Big|^{n_K/2} |\delta + it|^{-2} \,dt \nonumber\\
    &\le \int_{M}^\infty  \Big( 1 + \frac{1}{t} \Big)^{n_K/2} t^{-2} \,dt \nonumber\\
    &\le \int_1^{1 + M\inv} u^{n_K/2}\, du \nonumber\\
    &\le \frac{1}{n_K/2 + 1} \Big[(1 + M\inv)^{n_K/2 + 1} - 1 \Big].\label{eq: M integral}
\end{align}
Our bound on the sum of the two integrals is then
\begin{equation*}
    \frac{1}{n_K/2 + 1} \Big[ \frac{1}{\delta} (\delta + 1 + M)^{n_K/2 + 1}  + \Big(\frac{2\ell}{\epsilon }\Big)^{n_K/2 + 1} \Big( (1 + M\inv)^{n_K/2 + 1} - 1 \Big) \Big].
\end{equation*}
If we choose $M$ so that 
\begin{equation}\label{eq: M condition}
    \delta + 1 + M < \frac{2\ell}{\epsilon },
\end{equation}
we can bound the sum of the two integrals (using $\delta < 1$) by
\begin{equation}\label{eq: integral sum bound}
    \frac{1}{(n_K/2 + 1) \delta} \Big[\frac{2\ell}{\epsilon}(1 + M\inv)   \Big]^{n_K/2 + 1}.
\end{equation}
Choosing $M = \frac{1}{\epsilon}$ satisfies \eqref{eq: M condition}. Then, using $n_K \ge 1$ and recalling that $\ell = \lceil n_K/2 \rceil + 1 \le (n_K + 3)/2 \le 2n_K$, we can bound \eqref{eq: integral sum bound} by
\begin{equation*}
    \frac{1}{ \delta} \cdot \frac{4}{3} (2n_K)^{n_K/2}  \Big[\frac{2 }{\epsilon} (1 + \epsilon) \Big]^{n_K/2 + 1} .
\end{equation*}
The error term is then bounded by

\begin{align*}
    Z_\m(\delta) &\frac{1}{\pi} \Big( \frac{d_\chi}{ (2\pi)^{n_K}} \Big)^{1/2} \frac{1 + \delta}{ 1-  \delta} ( 1 + \delta\inv)^{n_K} e^{\delta \epsilon} x^\delta ( 1+ x^{-\delta/2}) \frac{1}{ \delta} \cdot \frac{4}{3} (2n_K)^{n_K/2} \Big[\frac{2 }{\epsilon } (1 + \epsilon) \Big]^{n_K/2 + 1} 
\end{align*}
as desired.
\end{proof}
We now sum over integral ideals in a given coset of the ray class group via \cref{lem: character sum}.

\begin{lem}\label{lem: ideal counting}
Let $\mathfrak{m} = \mathfrak{f}_{L/K}$, let $H = \ker(F_{L/K})$, and let $aH$ be a coset in $I(\m)/H$. Let $\mathfrak{n}$ be an integral ideal coprime to $\mathfrak{m}$. Then
\begin{equation*}
    \Big| \sum_{\substack{\mathfrak{a} \in aH \\ \n \mid \mathfrak{a}}} \phi\Big(\frac{\log \N(\mathfrak{a})}{\log x} \Big) - \frac{1}{[L:K]}\frac{\varphi_K(\mathfrak{m})}{\N(\mathfrak{m})}\frac{\kappa_K}{ \N(\n)}\cdot \log x \cdot F(-\log x)\Big| \leq E(x),
\end{equation*}
where $E(x)$ is as in \eqref{eq: E def}.
\end{lem}
\begin{proof}
By character orthogonality applied to the group $I(\m)/H$, we have that
\begin{align*}
    \sum_{\substack{\mathfrak{a} \in aH \\ \n \mid \mathfrak{a}}} \phi\Big( \frac{\log \N(\mathfrak{a})}{\log x} \Big) &=
    \frac{1}{[L:K]} \sum_{\chi(H) = 1} \overline{\chi}(aH) \sum_{\substack{\mathfrak{a}\\ \n \mid \mathfrak{a}}} \chi(\mathfrak{a})\phi\Big( \frac{\log \N(\mathfrak{a})}{\log x} \Big).
\end{align*}
Note that $\chi(\mathfrak{a}) = 0$ is defined if $\mathfrak{a}$ and $\mathfrak{m}$ are not coprime. Summing over each character using the previous lemma gives the desired result.
\end{proof}

\subsection{Bounding $V(z)$}\label{sec: Vz}
We proceed similarly to the proof of Lemma \ref{lem: character sum}. In particular, we bound $V(z)$ by a sum of test functions, and apply an inverse Laplace transform. Our main result is the following.
\begin{lem}\label{lem:vzbound}
Let $0 < \omega < 1/2$ and let $ \eta, \epsilon \in (0, 1/2)$. Set $z = x^\omega$. Set
\begin{align}
    e^{c_{11}} &= \frac{2^{1/2}}{e^{1/2} \pi} \Big( \frac{\eta + 1}{\eta(1 - \eta)} \Big)e^{\epsilon\omega/2} \Big[\frac{1 }{\epsilon \omega}(1 + \epsilon \omega) (1 + e^\epsilon) \Big]^{3/2}\nonumber \\
    e^{c_{12}} &= \frac{e^{1/2}}{\pi^{1/2}}\Big(\frac{\eta + 1}{\eta}\Big) \Big[\frac{1 }{\epsilon \omega}(1 + \epsilon \omega) (1 + e^\epsilon) \Big]^{1/2} \label{eq: c1c2 def}.
\end{align}
Denote by $\kappa_K$ the residue of $\zeta_K(s)$ at $ s= 1$.
If
\begin{align}\label{eq:vzrange}
    x \ge e^{\epsilon} \Big( e^{c_{11}} e^{c_{12} n_K} {n_K}^{n_K/2} D_K^{1/2} \Big)^{\frac{2}{(1-\eta)\omega}},
\end{align}
then
\begin{equation}
    V(z) \ge \frac{\kappa_K \omega}{2}  \log x.
\end{equation}
\end{lem}
\begin{proof}
Choose $z'$ such that 
\begin{equation*}
    \frac{\log z}{\log z'} \ge 1 + \frac{\epsilon}{\log x}
\end{equation*}
This arises from the support condition $\supp \phi \subset [ \frac{1}{2} - \frac{\epsilon}{\log x}, 1 + \frac{\epsilon}{\log x} ]$. Rearranging gives
\begin{equation*}
    \log z' \le \frac{\log z}{1 + \frac{\epsilon}{\log x}}.
\end{equation*}
Since $(1 + \epsilon/\log x)\inv > 1 - \epsilon/\log x$, we can set 
\begin{equation}\label{eq: z'def}
    z' = z^{1- \epsilon/\log x} = z e^{-\epsilon\omega} = (x e^{-\epsilon})^{\omega}.
\end{equation}

We have the inequality
\begin{equation*}
    V(z) := \sum_{\N(\mathfrak{a})\leq z}\frac{1}{\N(\mathfrak{a})} \geq 1 + \sum_\mathfrak{a}\frac{1}{\N(\mathfrak{a})}\phi\Big(\frac{\log \N(\mathfrak{a})}{\log z'}\Big).
\end{equation*}
By Laplace inversion, since $0 < \eta \leq 1/2$, we have that
\begin{align*}
    \sum_\mathfrak{a}\frac{1}{\N(\mathfrak{a})}\phi\Big(\frac{\log \N(\mathfrak{a})}{\log z'}\Big) &= \frac{\log z'}{2\pi i}\int_{ \eta-i\infty}^{ \eta + i\infty}\zeta_K(s + 1)F(-s\log z')\,ds \nonumber \\
    &=  \kappa_K F(0)\cdot \log z'  + \frac{\log z'}{2\pi i}\int_{-1+ \eta- i\infty}^{-1 + \eta +i\infty} \zeta_K(s+1) F(-s \log z')\,ds.\label{eq: Laplace transform}
\end{align*}
Using the definition of $z'$ and \cref{lem:test-fn}(e), we can bound the contribution of the residue from below as
\begin{align*}
    \kappa_K F(0)\cdot \log z' &= \kappa_K \omega (\log x- \epsilon)  \Big( \frac{1}{2 } + \frac{\epsilon}{\log x} \Big) \\
    &\ge  \frac{\kappa_K\omega}{2}  \log x.
\end{align*}
As before, we break up the integral into two parts, the first with $|t| \le M$ and the second with $|t| > M$. In both regions we use \cref{lem: Lfunctionbound} to bound $\zeta_K(s+1)$. We use \cref{lem:test-fn}(g) to bound the integrand with $|t| \le M$, and \cref{lem:test-fn}(f) for $|t| > M$. Applying these bounds gives
\begin{align}
    & \Big|\int_{-1 + \eta -i\infty}^{-1 + \eta +i\infty} \zeta_K(s+1)F(-s \log z') \,ds \Big| \\ \le ~&\Big(\frac{1 + \eta}{1- \eta}\Big) (1 + \eta \inv)^{n_K} \Big(\frac{D_K}{ (2\pi)^{n_K}} \Big)^{1/2} 2 e^{\epsilon \omega/2} \frac{(z')^{(-1 + \eta)/2}}{\log z'}   \nonumber \\
    &\times \Big[  \int_{0}^M  \frac{|\eta + 1 + it |^{n_K/2}}{|\eta + it|  } \Big( \frac{2}{2 - |\eta + it| \epsilon \omega/ \ell} \Big)^\ell dt + \Big( \frac{\ell(1 + e^{\epsilon})}{ \epsilon} \Big)^\ell \int_{M}^\infty  \frac{|\eta + 1 + it |^{n_K/2}}{|\eta + it|^{\ell + 1} } \,dt \Big]. \label{eq: two integrals}
\end{align}

To bound these integrals, we set $\ell = \lceil n_K/2 \rceil + 1$, so that $n_K/2 \le \ell -1$. We bound the first integral by
\begin{align*}
    \int_{0}^M  \frac{|\eta + 1 + it |^{n_K/2}}{|\eta + it|  } \Big( \frac{2}{2 - |\eta + it| \epsilon \omega/ \ell} \Big)^\ell \,dt &\le \frac{1}{\eta} \Big( \frac{2}{2 - (\eta + M) \epsilon \omega/ \ell} \Big)^\ell \int_{0}^M  (\eta + 1 + t )^{\ell -1} \,dt \\
    &\le \frac{1}{\eta} \Big( \frac{2}{2 - (\eta + M) \epsilon \omega/ \ell} \Big)^\ell \frac{1}{\ell} (\eta + 1 + M)^\ell.
\end{align*}
We bound the second integral using the fact that $|\eta +1 +it|^{n_K/2} \le |\eta + 1 +it|^{\ell -1}$ and then integrating as in \eqref{eq: M integral}
\begin{align*}
    \int_{M}^\infty  \frac{|\eta + 1 + it |^{n_K/2}}{|\eta + it|^{l + 1} } dt  &\le \int_{M}^\infty  \frac{|\eta + 1 + it |^{\ell - 1}}{|\eta + it|^{l + 1} } dt \\
    &\le \int_{M}^\infty  \Big| 1 + \frac{1}{\eta + it} \Big|^{\ell - 1} |\eta + it|^{-2} dt \\
    &\le \frac{1}{\ell} \Big[(1 + M\inv)^\ell - 1 \Big].
\end{align*}
So overall, our bound on the sum of the two integrals in \eqref{eq: two integrals} is 
\begin{equation*}
    \frac{1}{\ell} \Big[ \frac{1}{\eta} \Big( \frac{2(\eta + 1 + M)}{2 - (\eta + M) \epsilon \omega/ \ell} \Big)^\ell  + \Big(\frac{\ell(1+ e^\epsilon)}{\epsilon \omega}\Big)^\ell \Big( (1 + M\inv)^\ell - 1 \Big) \Big]
\end{equation*}
We want to choose $M$ small enough so that 
\begin{equation}\label{eq: Vz M condition}
    \frac{2(\eta + 1 + M)}{2 - (\eta + M) \epsilon \omega/ \ell} < \frac{\ell(1+ e^\epsilon)}{\epsilon \omega}
\end{equation}
and we can then bound the entire quantity (using $\eta < 1$) by
\begin{equation}\label{eq: M bound}
    \frac{1}{\ell \eta} \Big[\frac{\ell}{\epsilon\omega}(1 + M\inv) (1 + e^\epsilon)  \Big]^\ell.
\end{equation}
Setting $M = \frac{1}{\epsilon\omega}$ satisfies \eqref{eq: Vz M condition}. Then, by recalling that $\ell = \lceil n_K/2\rceil + 1 \le 2n_K$ and using that $n_K \le e^{n_K-1}$, we can bound \eqref{eq: M bound} by
\begin{equation*}
    \frac{1}{ \eta} n_K^{n_K/2} 2^{n_K/2} 2^{1/2} e^{n_K/2} e^{-1/2} \Big[\frac{1 }{\epsilon \omega}(1 + \epsilon \omega) (1 + e^\epsilon) \Big]^\ell
\end{equation*}
So in all, our bound for the integral in \eqref{eq: Laplace transform} is 
\begin{equation*}
    (z')^{(-1 + \eta)/2} e^{c_{11}} e^{c_{12} n_K}  n_K^{n_K/2} D_K^{1/2}
\end{equation*}
where
\begin{align*}
    e^{c_{11}} &= \frac{2^{1/2}}{e^{1/2} \pi}  \frac{\eta + 1}{\eta(1 - \eta)} e^{\epsilon\omega/2} \Big[\frac{1 }{\epsilon \omega}(1 + \epsilon \omega) (1 + e^\epsilon) \Big]^{3/2} \\
    e^{c_{12}} &= (1 + \eta\inv) (2\pi)^{-1/2} 2^{1/2} e^{1/2} \Big[\frac{1 }{\epsilon \omega}(1 + \epsilon \omega) (1 + e^\epsilon) \Big]^{1/2}.
\end{align*}
We need the error term to be less than 1, which is satisfied when 
\begin{align*}
    z' \ge \Big( e^{c_{11}} e^{c_{12} n_K} n_K^{n_K/2} D_K^{1/2} \Big)^{2/(1 - \eta)}.
\end{align*}
Recalling that $z' = (xe^{-\epsilon})^\omega$ completes the proof.
\end{proof}

%% file: 4-selberg-sieve.tex
Let $L/K$ be an abelian extension of number fields. In this section we prove an upper bound for $\pi_C(x, L/K)$ via the Selberg sieve. 
Our implementation closely follows the classical method of proof of the Brun--Titchmarsh theorem via the Selberg sieve (see \cite{cojocaru2006} for one such proof). We apply the analogue for number fields outlined in \cite{Wei}. Our main result is the following.
\begin{lem}\label{lem: selberg sieve}
Let $L/K$ be an abelian extension of number fields and $C$ a conjugacy class (which will necessarily be a single element) of $\Gal(L/K)$. Let $\pi_C(x, L/K)$ be as in \eqref{eq: piC def}, $E(x)$ be as in \eqref{eq: E def}, and $0 < \omega, \eta, \epsilon, \gamma < 1/2$. If $x$ satisfies the range condition \eqref{eq:vzrange}, then
\begin{equation}\label{eq:picab1}
    \pi_C(x, L/K) \leq 2.52 \frac{n_K \sqrt{x}}{\log x} + \frac{2e^\epsilon x}{\omega [L:K]\log x} + E(x)\Big((1 + \gamma\inv)^{n_K}x^{\omega(1 + \gamma)}\Big)^2,
\end{equation}
\end{lem}

\begin{proof}
Let $aH \coloneqq F_{L/K}^{-1}(C)$ and recall \eqref{eq: translation to cosets}. We would then like to bound
\begin{equation*}
    \pi_C(x, L/K) \ = \ \pi_{aH}(x, L/K) \ = \ \sum_{\substack{\mathfrak{p} \text{ unramified} \\ \mathfrak{p} \in aH \\ \N_{K/\mathbb{Q}}(\mathfrak{p}) \leq x}} 1.
\end{equation*}
To begin, we split the sum into two pieces as follows.
\begin{align}
    \pi_{aH}(x,L/K) \leq   \sum_{\substack{\mathfrak{p} \in aH \\ \N(\mathfrak{p}) \leq x}} 1 =  \sum_{\substack{\mathfrak{p} \in aH \\ \N(\mathfrak{p}) \leq \sqrt{x}}} 1 +   \sum_{\substack{ \mathfrak{p} \in aH \\ \sqrt{x} < \N(\mathfrak{p}) \leq x}} 1. \label{eq: Selberg sum split}
\end{align}
Since at most $n_K$ prime ideals $\mathfrak{p}$ of $K$ lie over a given rational prime $p$, we can use the universal bound $\sum_{p \leq x} 1 \leq 1.26 \frac{x}{\log x}$ from Corollary 1 of \cite{rosser1962} to bound the first sum in \eqref{eq: Selberg sum split} by
\begin{equation}\label{eq:sqrtbound}
    \sum_{\substack{\mathfrak{p} \in aH \\ \N(\mathfrak{p}) \leq \sqrt{x}}} 1 \leq  \sum_{p \leq \sqrt{x}} \sum_{\mathfrak{p} \mid \langle p \rangle} 1 \leq n_K \sum_{p \leq \sqrt{x}} 1 \leq 2.52 \frac{n_K \sqrt{x}}{\log x}.
\end{equation}

We bound the second term via the Selberg sieve. For an integral ideal $\mathfrak{a}$, we define the quantity $\mathcal{P}^-(\mathfrak{a})$ to be the least norm of a prime ideal factor of $\mathfrak{a}$. Let $0 < z < \sqrt{x}$. Let $S_z$ be the set of those integral ideals satisfying $\mathcal{P}^-(\mathfrak{a}) > z$; that is, integral ideals with no prime ideal factor having norm $\leq z$. It follows that
\begin{equation*}
    \sum_{\substack{ \mathfrak{p} \in aH \\ \sqrt{x} < \N(\mathfrak{p}) \leq x}} 1 \le  \sum_{\substack{\mathfrak{a} \in aH \cap S_z \\ \sqrt{x} < \N(\mathfrak{a})\leq x}} 1.
\end{equation*}
For the test function $\phi$ defined in \cref{lem:test-fn}, we have that $\phi(t) = 1$ for $t \in [\frac{1}{2}, 1]$ and $\phi$ is nonnegative, so
\begin{equation*}
    \sum_{\substack{\mathfrak{a} \in aH \cap S_z\\ \sqrt{x} < \N(\mathfrak{p}) \leq x}} 1 \leq \sum_{\mathfrak{a} \in aH \cap S_z} \phi\Big( \frac{\log \N(\mathfrak{a})}{\log x} \Big).
\end{equation*}
Now, consider any real-valued function $\lambda$ on integral ideals of $\mathcal{O}_K$ satisfying
\begin{enumerate}
    \item $\lambda(\mathcal{O}_K) = 1$;
    \item $\lambda(\mathfrak{b}) = 0$ unless $\mathfrak{b}$ is squarefree, coprime to $\mathfrak{m}$ and $\N(\mathfrak{b}) \leq z$;
    \item $|\lambda(\mathfrak{b})| \leq 1$ for all $\mathfrak{b}$.
\end{enumerate}
We will eventually specify the choice of such a function.

Denote by $[\mathfrak{b}_1, \mathfrak{b}_2]$ and $(\mathfrak{b}_1,\mathfrak{b}_2)$ the least common multiple and greatest common divisor of two integral ideals $\mathfrak{b}_1$, $\mathfrak{b}_2$. Also denote by $\mu_K$ the M\"{o}bius function on integral ideals associated to the number field $K$, which is defined for an integral ideal $\mathfrak{a}$ of $\mathcal{O}_K$ as
\begin{align*}
    \mu_K(\mathfrak{a}) := \begin{cases}
        (-1)^{\omega(\mathfrak{a})} & \mathfrak{a} \text{ is squarefree},\\
        0 & \mathfrak{a} \text{ is not squarefree}.
    \end{cases}
\end{align*}

Note that if $\mathfrak{a} \in S_z$, then $\mathfrak{b} \mid \mathfrak{a}$ implies $\mathfrak{b} = \mathcal{O}_K$ or $\N(\mathfrak{b}) > z$, which by conditions (a) and (b) implies that $\mathfrak{b} = \mathcal{O}_K$ or $\lambda(\mathfrak{b}) = 0$. Hence, we have

\begin{align}
    \sum_{\mathfrak{a} \in aH \cap S_z} \phi\Big( \frac{\log \N(\mathfrak{a})}{\log x} \Big)
    &\leq \sum_{\mathfrak{a} \in aH} \phi\Big( \frac{\log \N(\mathfrak{a})}{\log x} \Big) \Big( \sum_{\mathfrak{b} \mid \mathfrak{a}} \lambda(\mathfrak{b}) \Big)^2\nn \\
    &= \label{eq:lambdalambda} \sum_{\mathfrak{b}_1,\mathfrak{b}_2} \lambda(\mathfrak{b}_1) \lambda(\mathfrak{b}_2) \sum_{\substack{\mathfrak{a} \in aH \\ [\mathfrak{b}_1,\mathfrak{b}_2] \mid \mathfrak{a}}} \phi\Big( \frac{\log \N(\mathfrak{a})}{\log x} \Big).
\end{align}
Now by Lemma \ref{lem: ideal counting}, we have for $\mathfrak{n}, \mathfrak{m}$ coprime that
\begin{equation*}
    \sum_{\substack{\mathfrak{a} \in aH \\ \mathfrak{n} \mid \mathfrak{a}}} \phi \Big( \frac{\log \N(\mathfrak{a})}{\log x} \Big) \leq \frac{1}{[L:K]}\frac{\varphi_K(\mathfrak{m})}{\N(\mathfrak{m})}\frac{\kappa_K}{ \N(\n)}(\log x) \cdot F(-\log x) + |E(x)|,
\end{equation*} 
If $[\b_1, \b_2]$ is not coprime to $\m$, then $\lambda(\b_1)\lambda(\b_2) = 0$ by condition (b), so \eqref{eq:lambdalambda} is bounded by
\begin{align*}
&\sum_{\mathfrak{b}_1,\mathfrak{b}_2} \lambda(\mathfrak{b}_1) \lambda(\mathfrak{b}_2) \Big( \frac{1}{[L:K]}\frac{\varphi_K(\mathfrak{m})}{\N(\mathfrak{m})}\frac{\kappa_K}{ \N([\mathfrak{b}_1,\mathfrak{b}_2])}(\log x) \cdot F(-\log x) + |E(x)|\Big).
\end{align*}
Since $\mathfrak{b}_1 \mathfrak{b}_2 = (\mathfrak{b}_1,\mathfrak{b}_2)[\mathfrak{b}_1,\mathfrak{b}_2]$, we can rewrite this as
\begin{align*}
&\sum_{\mathfrak{b}_1,\mathfrak{b}_2} \lambda(\mathfrak{b}_1) \lambda(\mathfrak{b}_2)\Big( \frac{1}{[L:K]}\frac{\varphi_K(\mathfrak{m})}{\N(\mathfrak{m})}\frac{\kappa_K \N((\mathfrak{b}_1,\mathfrak{b}_2))}{ \N(\mathfrak{b}_1)\N(\mathfrak{b}_2)}(\log x) \cdot F(-\log x) + E(x) \Big).
\end{align*}
Applying the identity $\N(\mathfrak{a}) = \sum_{\mathfrak{b}\mid \mathfrak{a}}\varphi_K(\mathfrak{b})$ and $|\lambda(\mathfrak{b})|\leq 1$ yields
\begin{align}
&\frac{\kappa_K(\log x) F(-\log x) }{[L:K]}\frac{\varphi_K(\mathfrak{m})}{\N(\mathfrak{m})}\sum_{\mathfrak{b}_1,\mathfrak{b}_2} \frac{\lambda(\mathfrak{b}_1) \lambda(\mathfrak{b}_2)}{ \N(\mathfrak{b}_1)\N(\mathfrak{b}_2)}\sum_{\mathfrak{a}\mid \mathfrak{b}_1,\mathfrak{a}\mid \mathfrak{b}_2}\varphi_K(\mathfrak{a}) + \sum_{\mathfrak{b}_1,\mathfrak{b}_2} \Big|\lambda(\mathfrak{b}_1) \lambda(\mathfrak{b}_2) \Big| E(x) \nn \\
&\leq \label{eq:ssstep} \frac{\kappa_K(\log x) F(-\log x) }{[L:K]}\frac{\varphi_K(\mathfrak{m})}{\N(\mathfrak{m})}\sum_\mathfrak{a}\varphi_K(\mathfrak{a}) \Big(\sum_{\mathfrak{a}\mid \mathfrak{b}} \frac{\lambda(\mathfrak{b})}{\N(\mathfrak{b})}\Big)^2 + E(x)\Big(\sum_{\substack{\mathfrak{b}\\ \N(\mathfrak{b})\leq z}}1\Big)^2.
\end{align}
By considering the M\"{o}bius inversion pair
\begin{equation*}
    \xi(\mathfrak{a}) = \sum_{\mathfrak{b}\mid \mathfrak{a}}\frac{\lambda(\mathfrak{b})}{\N(\mathfrak{b})},\quad \frac{\lambda(\mathfrak{a})}{\N(\mathfrak{a})} = \sum_\mathfrak{b}\mu_K(\mathfrak{b})\xi(\mathfrak{a}\mathfrak{b}),
\end{equation*}
we seek to choose $\lambda$ (equivalently, to choose $\xi$) such that the quadratic form $\sum_\mathfrak{a}\varphi_K(\mathfrak{a})\xi(\mathfrak{a})^2$ is minimized. Using the same choice as in \cite[p. 73]{Wei} and recalling the definition of $V(z)$ in \eqref{eq: Vz def} yields the following bound from \cite[Lemma 3.6,~eq.~(iv)]{Wei}, which states that
\begin{equation*}
    \frac{\varphi_K(\mathfrak{m})}{\N(\mathfrak{m})} \sum_\mathfrak{a}\varphi_K(\mathfrak{a})\Big(\sum_{\mathfrak{a}\mid \mathfrak{b}}\frac{\lambda(\mathfrak{b})}{\N(\mathfrak{b})}\Big)^2 \leq \frac{1}{V(z)}.
\end{equation*}
 Substituting this into equation \eqref{eq:ssstep} yields the bound
\begin{equation*}
    \frac{\kappa_K(\log x)F(-\log x)}{[L:K]V(z)} + E(x)\Big(\sum_{N(\mathfrak{b})\leq z}1\Big)^2.
\end{equation*}
By \cref{lem:test-fn}(e) the above is bounded by
\begin{equation*}
    \frac{\kappa_K e^\epsilon x}{[L:K]V(z)} + E(x)\Big(\sum_{N(\mathfrak{b})\leq z}1\Big)^2.
\end{equation*}
Now, by the lower bound for $V(z)$ from \cref{lem:vzbound} (in the appropriate range for $x$ specified in that lemma, with $z = x^\omega$), the above is bounded by
\begin{equation*}
    \frac{2e^\epsilon x}{\omega [L:K]\log x} + E(x)\Big(\sum_{ N(\mathfrak{b})\leq z}1\Big)^2.
\end{equation*}
Applying Lemma \ref{lem: weakidealbound} and substituting $z = x^\omega$ yields a bound (for $0 < \gamma \leq 1$) of
\begin{equation}\label{eq:ssbound}
    \sum_{\substack{ \mathfrak{p} \in aH \\ \sqrt{x} < \N(\mathfrak{p}) \leq x}} 1 \le \frac{2e^\epsilon x}{\omega [L:K]\log x} + E(x)\Big((1 + \gamma\inv)^{n_K}x^{\omega(1 + \gamma)}\Big)^2.
\end{equation}
Combining \eqref{eq:sqrtbound} and \eqref{eq:ssbound} into \eqref{eq: Selberg sum split} gives the desired result.
\end{proof}

%% file: 5-BT.tex
In this section we prove \cref{lem: BT abelian}.
We begin by bounding the error terms in Lemma \ref{lem: selberg sieve}. We bound the first term in \eqref{eq:picab1} as
\begin{equation}
    2.52 \frac{n_K \sqrt{x}}{\log x} \leq \frac{x}{100[L:K]\log x}
\end{equation}
in the range 
\begin{equation}\label{eq: small prime range}
    x \geq (252n_K[L:K])^2.
\end{equation}
Next we consider the last term in \eqref{eq:picab1}. Let $E(x)$ be as in \eqref{eq: E def}. Then we define
\begin{equation}\label{eq: Cdeg def}
    C(\delta, \epsilon, \gamma) = E(x)(1 + \gamma\inv)^{2n_K} [L:K] x^{-\delta},
\end{equation}
which does not depend on $x$. We want to choose $x$ large enough so that
\begin{equation*}
    E(x)(1 + \gamma\inv)^{2n_K}x^{2\omega(1 + \gamma)}\leq \frac{x}{[L:K]\log x},
\end{equation*}
which is equivalent to
\begin{equation}\label{xrange}
    \frac{x^{1-\delta-2\omega(1 + \gamma)}}{\log x}\geq C(\delta, \epsilon, \gamma).
\end{equation}
By Lemma \ref{abc}, we find that if
\begin{equation}
    x \geq \Big( \frac{  C(\delta, \epsilon, \gamma )}{1 - \delta - 2\omega(1 + \gamma)} \Big)^{\nu},
    \end{equation}
where
\begin{equation}\label{eq: nu def}
    \nu = \frac{1}{1 - \delta - 2\omega(1 + \gamma)} \Big(1 + \Big(\frac{2}{\log(C(\delta,\epsilon,\gamma)/(1-\delta-2\omega(1 + \gamma)))}\Big)^{1/2}\Big),
\end{equation}
then \eqref{xrange} is satisfied.
Combining the bounds in this section with \cref{lem:vzbound} yields our first version of the Brun-Titchmarsh theorem for abelian extensions. 
\begin{lem}\label{lem: BT ugly}
Let $L/K$ be an abelian extension of number fields, let $C$ be a conjugacy class (which will necessarily be a single element) of $\Gal(L/K)$, and let $\pi_C(L/K)$ be the prime--counting function given in \eqref{eq: piC def}. Let $0 < \epsilon, \delta, \gamma, \eta, \omega < 1/2$ be such that $1 - \delta - 2\omega(1 + \gamma) > 0$, and let $c_{11}, c_{12}$ be as in \eqref{eq: c1c2 def}, $C(\delta, \epsilon, \gamma)$ be as in \eqref{eq: Cdeg def}, and $\nu$ be as in \eqref{eq: nu def}. If
\begin{equation}
    x \ge \max\Big( e^{\epsilon} \Big( e^{c_{11}} e^{c_{12} n_K} {n_K}^{n_K/2} D_K^{1/2} \Big)^{\frac{2}{(1-\eta)\omega}}, (252n_K[L:K])^2, \Big( \frac{  C(\delta, \epsilon, \gamma )}{1 - \delta - 2\omega(1 + \gamma)} \Big)^{\nu} \Big)
\end{equation}
then 
\begin{equation}
    \pi_C(x, L/K) \le \Big(1.01 + \frac{2e^\epsilon}{\omega} \Big) \frac{x}{[L:K] \log x}.
\end{equation}

\end{lem}
We state the above version of the Brun-Titchmarsh theorem as it allows us flexibility in choosing parameters when the field extension $L/K$ is known, as in Section \ref{sec: LT}. For $L/K$ a general abelian extension of number fields, we derive an explicit form of \cref{lem: BT ugly} by fixing parameters. 
Note that different choices of parameters may improve the dependence of the range on certain field constants (such as $n_K, D_K, \mathcal{Q}$) at the expense of others. 
\begin{proof}[Proof of \cref{lem: BT abelian}]
Let $\delta = 1/10$, $\eta = 1/21$, $\epsilon = \omega = 1/4$, and $\gamma = 1/5$. We wish to obtain a range for $x$ in terms of $n_K, D_K$, and $\mathcal{Q}$ via \cref{lem: BT ugly}. To do so, we use \cref{lem: degree bound} with the parameter $\epsilon = 1$ to bound $[L:K]$ and \cref{lem: Pmd bound} with the parameters $M = 400$ and $b = (2\log(1+M^{-\delta}))\inv$ to bound $Z_\m(\delta)$. Doing so gives
\begin{align}
    C(\delta, \epsilon, \gamma) &\le e^{7.36} e^{22.85n} (D_K\mathcal{Q})^{2} n^{n/2}.
\end{align}
Applying these bounds to \eqref{eq: nu def} gives $\nu \le 4.189$. Now, combining the bounds given in \cref{lem: BT ugly} and substituting our parameters gives the desired result.
\end{proof}

%% file: 7-LT.tex
In this section we apply our Brun-Titchmarsh results to study coefficients of modular forms. Let $f(z) \in \mathcal{S}_{k_f}^{\text{new}}(\Gamma_0(N_f))$ be a non-CM cusp form of even weight $k_f \geq 2$, level $N_f$, trivial nebentypus, and integral coefficients $a_f(n)$ which is an eigenform for all Hecke operators and Atkin--Lehner operators. Let
\begin{equation*}
    \pi_f(x,a) \coloneqq \#\{ p \leq x \text{ prime},~p \nmid N_f, \text{ and } a_f(p) = a\}.
\end{equation*}
We use a sieving procedure to bound $\pi_f(x, a)$ by a sum of auxiliary prime-counting functions $\pi_f(x, a, \ell)$ (for a suitably chosen collection of primes $\ell$) which count coefficients $a_f(p)$ congruent to $a \Mod{\ell}$, with some additional conditions. The projection onto $\mathbb{F}_\ell$ of the $\ell$-adic Galois representation $\rho_{f, \ell}$ attached to $f$ allows us to express these conditions as a condition on the image of $\Frob_p$ under the residual representation. This allows the count of rational primes $p$ with given $a_f(p)$ to be compared with the count of prime ideals with prescribed Frobenius class in a certain abelian extension of number fields. This extension is a subextension of the fixed field of $\ker \rho_{f,\ell}$ in $\overline{\mathbb{Q}}$, through which $\rho_{f, \ell}$ factors. This enables us to compute an explicit upper bound for $\pi_f(x,a)$ using \cref{lem: BT ugly}. 

\subsection{Sieving $\pi_f(x,a)$ by primes}
For a prime number $p$, define
\begin{equation}
    \omega_p = (a_f(p)^2 - 4p^{k_f - 1})^{1/2}.
\end{equation}
We know from Deligne's proof of the Weil conjectures that $|a_f(p)| \leq 2p^{(k_f - 1)/2}$ for all $p$, so $\mathbb{Q}(\omega_p)$ is an imaginary quadratic extension of $\mathbb{Q}$. For an odd prime $\ell$, set
\begin{equation*}
    \pi_f(x,a;\ell) = \#\Big\{ p \leq x \text{ prime},~ p \nmid N_f,~ a_f(p) \equiv a \Mod{\ell},~\Big( \frac{a^2 - 4p^{k_f - 1}}{\ell} \Big) = + 1\Big\}.
\end{equation*}
Here, $( \frac{\cdot}{\ell})$ is the Legendre symbol; in other words the latter condition says $\ell$ splits in $\mathbb{Q}(\omega_p)$. Our sieve result is the following lemma adapted from~\cite[Lemma~4.1]{Wan}. 

\begin{lem}\label{lem:pi_f-sieve}
Let $x, t > 0$ be integers, and let $\ell_1 < \ell_2 < \cdots < \ell_t$ be $t$ odd primes, each less than $x$. Assume furthermore that $\mathrm{gcd}(\ell_j - 1, k_f - 1) = 1$ for all $j = 1,\dots,t$. Then
\begin{equation}\label{eq:pi_f-sieve}
    \pi_f(x,a) \leq \sum_{j = 1}^{\ell} \pi_f(x,a;\ell_j) + \sqrt{2} \cdot\frac{x} {2^{t/2}} + \sqrt{2} \cdot 2^{t/2} \ell_t^t + 2^t +  \frac{70}{2^t} .
\end{equation}
\end{lem}
\begin{proof}
Let $\ell_1 < \ell_2 < \cdots < \ell_t$ be as above, and let $M_t(x)$ be the number of primes $p \leq x$ such that $\Big(\frac{a^2 - 4p^{k_f-1}}{\ell_j}\Big) \neq +1$ for all $j = 1,\dots,t$, that is, the number of primes $p \leq x$ that are not counted in any of the $\pi_f(x,a;\ell_j)$. Then
\begin{equation}\label{eq:pi_f-sieve-M}
    \pi_f(x,a) \leq \sum_{j = 1}^t \pi_f(x,a,\ell_j) + M_t(x).
\end{equation}
To estimate $M_t(x)$, we introduce a weight function $w_t(p)$ on the set of primes. Given a prime $p$, suppose that $\Big(\frac{a^2 - 4p^{k_f-1}}{\ell}\Big) = +1,0,-1$ for, respectively, $t_1(p)$, $t_2(p)$, $t_3(p)$ of the primes $\ell_j$. So $t_1 + t_2 + t_3 = t$. Define
\begin{align*}
    w_t(p) = \frac{1}{2^t} \prod_{j = 1}^t \Big( 1 - \Big( \frac{a^2 - 4p^{k_f-1}}{\ell_j} \Big) \Big) = \begin{cases}
    0 &\text{ if } t_1(p) \neq 0, \\
    2^{-t_2(p)} &\text{ if } t_1(p) = 0,
    \end{cases}
\end{align*}
and let
\begin{equation*}
    W_t(x) = \sum_{p \leq x} w_t(p).
\end{equation*}
Then
\begin{align*}
    W_t(x) &\leq \frac{1}{2^t} \sum_{n \leq x} \prod_{j = 1}^t \Big( 1 - \Big( \frac{a^2 - 4n^{k_f-1}}{\ell_j} \Big) \Big) \\
    &= \sum_{n \leq x} \Big( \frac{1}{2} \Big)^t +  \max_{\substack{d \geq 2 \\ d \mid L_t}}~ \Big| \sum_{n \leq x} \Big( \frac{a^2 - 4n^{k_f-1}}{d} \Big) \Big|,
\end{align*}
where $L_j = \ell_1 \cdots \ell_j$ for $j = 1,\dots,t$, so $L_t \leq \ell_t^t$. The last sum is a character sum mod $d$ since we assume that $\mathrm{gcd}(\ell_j - 1, k_f - 1) = 1$ for all $j$. We trivially bound this character sum above by $d \leq L_t$. 
Thus,
\begin{equation*}
    W_t(x) \leq \frac{x}{2^t} + \ell_t^t.
\end{equation*}
Let $M_t'(x)$ be the number of primes $p \leq x$ such that $t_2(p) \geq t/2$ (that is, $\Big(\frac{a^2 - 4p^{k_f-1}}{\ell_j}\Big) = 0$ 
So
\begin{equation*}
    M_t(x) - M_t'(x) \leq 2^{\lceil t/2 \rceil}W_t(x) \leq \sqrt{2} \cdot 2^{t/2} W_t(x).
\end{equation*}
Arguing via the Chinese remainder theorem, appealing again to our assumption that $\text{gcd}(\ell_i - 1,k_f - 1) = 1$, we have that
\begin{equation*}
    M_t'(x) \leq \binom{t}{\lceil t/2 \rceil} \Big( \frac{x}{\ell_1 \cdots \ell_{\lceil t/2 \rceil }} + 1 \Big) \leq 2^t \frac{x}{\ell_1 \cdots \ell_{\lceil t/2 \rceil}} + 2^t.
\end{equation*}
Therefore
\begin{equation}\label{eq:M_t-bound}
    M_t(x) \leq \sqrt{2} \cdot \frac{x}{2^{t/2}} + \sqrt{2} \cdot 2^{t/2} \ell_t^t + 2^t \frac{x}{L_{\lceil t/2 \rceil}} + 2^t.
\end{equation}
Since $\ell_j \geq 2^4$ except possibly for the first $5$ odd primes, we have that
\begin{equation*}
    L_{\lceil t/2 \rceil} \geq \frac{3 \cdot 5 \cdot 7 \cdot 11 \cdot 13}{16^5} \cdot \prod_{i = 1}^{\lceil t/2 \rceil} 16 \geq 70^{-1} \cdot 4^t,
\end{equation*}
and inserting \eqref{eq:M_t-bound} into \eqref{eq:pi_f-sieve-M} gives the lemma. 
\end{proof}

For our application, we state the following special case of~\cref{lem:pi_f-sieve}. 

\begin{cor}\label{cor:pi_f-sieve}
Let $r, x> 1$, let $t = \lceil (2r/\log 2) \log \log x \rceil$, and let $\ell_1 < \ell_2 < \cdots < \ell_t$ be $t$ odd primes, each less than $\exp( \frac{\log x}{2t} )$. Assume furthermore that $\mathrm{gcd}(\ell_j - 1, k_f - 1) = 1$ for all $j = 1,\dots,t$. Then
\begin{align}\label{eq:cor:pi_f-sieve}
    \pi_f(x,a) &\leq \Big( \frac{2r}{\log 2} \log \log x + 1 \Big) \max_{1 \leq t \leq j} \pi_f(x,a;\ell_j) \nonumber\\&\phantom{\,\,\,\,\,\,\,\,}+ \frac{x\sqrt{2}}{(\log x)^r} + \frac{2x^{1/2}}{(\log x)^r} + 2 (\log x)^{2r} + 35.
\end{align}
\end{cor}
\begin{proof}
Observe that we have
\begin{equation*}
    \frac{2r}{\log 2} \log \log x \leq t < \frac{2r}{\log 2} \log \log x + 1
\end{equation*}
with this choice of $t$, and use this to bound all terms in \eqref{eq:pi_f-sieve} from above.
\end{proof}

\subsection{Reduction to a Chebotarev problem}\label{sec: red cheb}
Let $\ell$ be any odd prime, and let $\mathbb{F}_{\ell}$ be the field of $\ell$ elements. For any prime $p$, let $\Frob_p$ be the Frobenius automorphism of $\Gal(\overline{\mathbb{Q}}/\mathbb{Q})$ at $p$. By a theorem of Deligne, one associates to each newform $f \in \mathcal{S}_{k_f}^{\text{new}}(\Gamma_0(N_f))$ an $\ell$-adic representation, and, by projection onto $\mathbb{F}_{\ell}$, a representation
\begin{equation*}
    \rho_{f,\ell}: \Gal(\overline{\mathbb{Q}}/\mathbb{Q}) \to \mathrm{GL}_2(\mathbb{F}_{\ell}),
\end{equation*}
which is unramified outside $N_f \ell$, and such that for all primes $p \nmid N_f \ell$, we have that $\tr(\rho_{f,\ell}(\Frob_p)) \equiv a_f(p) \Mod{\ell}$ and $\det(\rho_{f,\ell}(\Frob_p)) \equiv 4p^{k_f - 1} \Mod{\ell}$.

Let $L = L_{\ell}$ be the subfield of $\overline{\mathbb{Q}}$ fixed by $\ker \rho_{f,\ell}$. It is known that there exists $\ell_0$ depending on $f$ such that, if $\ell > \ell_0$, then $L/\mathbb{Q}$ is a Galois extension, unramified outside of $N_f \ell$, whose Galois group is $G = \{ g \in \mathrm{GL}_2(\mathbb{F}_{\ell}) \mid \det(g) \in (\mathbb{F}_{\ell}^\times)^{k_f-1}\}$. Hence, if $\mathrm{gcd}(\ell - 1, k_f - 1) = 1$, then $G = \mathrm{GL}_2(\mathbb{F}_{\ell})$; that is, the representation $\rho_{f,\ell}$ is surjective. 

From now on we let $\ell > \ell_0$ be a prime number with $\mathrm{gcd}(\ell -1,k_f - 1) = 1$, and let 
\begin{equation*}
    C = \{ A \in G \mid \tr(A) \equiv a\Mod{\ell} \text{ and } \tr(A)^2 - 4 \det(A) \in \mathbb{F}_{\ell} \text{ is a nonzero square}\}.
\end{equation*}
The latter condition is equivalent to saying that the matrix $A$ has distinct eigenvalues in $\mathbb{F}_{\ell}$. The set $C$ is stable under conjugation in $G$, and can be expressed as
\begin{equation*}
    C = \bigcup_{\gamma \in \Gamma} C_G(\gamma)
\end{equation*}
where
\begin{equation*}
    \Gamma = \Big\{ \begin{pmatrix} \alpha & 0 \\ 0 & \beta \end{pmatrix} \in G ~\bigg|~ \alpha \neq \beta,~\alpha + \beta \equiv a \Mod{\ell}\Big\},
\end{equation*}
and $C_G(\gamma)$ denotes the conjugacy class in $G$ that contains $\gamma$.
Let $B$ denote the Borel subgroup of upper triangular matrices in $G$, and let $L^B$ be the subfield of $L$ fixed by $B$. Let $H$ be the subgroup of $B$ consisting of matrices whose eigenvalues are both equal, and let $L^H$ be the subfield of $L$ fixed by $H$. Let $U$ be the subgroup of unipotent elements in $B$, and let $L^U$ be the subfield of $L$ fixed by $U$. We have
\begin{equation*}
    \{1\} \subset U \subset H \subset B \subset G = \mathrm{GL}_2(\mathbb{F}_{\ell}).
\end{equation*}
Note that $U$ and $H$ are normal subgroups of $B$ and that $B/U$, $B/H$ are abelian. We have the tower of field extensions
\begin{equation*}
    \mathbb{Q} \subset L^B \subset L^H \subset L^U \subset L,
\end{equation*}
with $L^U/L^B$ and $L^H/L^B$ both abelian extensions. Finally, let $C'$ be the image of $C \cap B$ in $B/U$ and let $C''$ be the image of $C \cap B$ in $B/H$. 

Before stating the reduction to counting prime ideals in conjugacy classes of Galois groups, we collect some basic facts about the cardinality of these groups and conjugacy classes. 

\begin{prop}\label{prop:about-subgroups}
We have $|C'| \leq \ell$, 
\begin{equation*}
        \frac{|C'|}{|B/U|} \leq \frac{1}{\ell - 1} \quad \text{and} \quad \frac{|C''|}{|B/H|} = \frac{1}{\ell - 1},
\end{equation*}
and $[L^B:\mathbb{Q}] = \ell + 1$, $[L^H:L^B] = \ell - 1$, and $[L^U:L^H] = \ell$.
\end{prop}
\begin{proof}
By \cite[Proof of Lemma 4.4]{Zyw}, we have $|G| = (\ell - 1)^2(\ell + 1) \ell$, $|B| = (\ell - 1)^2 \ell$, $|H| = (\ell - 1) \ell$, $|U| = \ell - 1$, $|C'| \leq \ell$, and $|C''| =1$.
\end{proof}

We now compare $\pi_f(x,a;\ell)$ to a prime-counting function of the form $\pi_C(x, L/K)$. Namely, for all $a \in \mathbb{Z}$, we will bound $\pi_f(x,a;\ell)$ in terms of $\pi_{C'}(x, L^U/L^B)$. However, when $a \equiv 0 \Mod{\ell}$, the conjugacy class $C$ consists of matrices of zero trace over $\mathbb{F}_{\ell}$, which allows us to work with the subextension $L^H/L^B$ and its corresponding conjugacy class $C''$. In view of this we state two versions of the reduction step, one applicable to all $a \in \mathbb{Z}$ and another specialized to the case $a \equiv 0 \Mod{\ell}$.

For a given modular form, we may also know additional congruence relations which its Fourier coefficients must satisfy. For instance, say we know that if $a_f(p) = a$, then $p$ lies in some fixed set $S$ of residue classes modulo $q \geq 1$. This condition imposes further conditions on the Artin symbol; via our reduction, it translates to counting primes in a compositum with a $q$-cyclotomic extension (that is, $L^U(\zeta_q)/L^B$ or $L^H(\zeta_q)/L^B$).

For $L/K$ a Galois extension of number fields with Galois group $G$, and $C \subseteq G$ a conjugacy class, we set
\begin{equation*}
    \widetilde{\pi}_{C}(x,L/K) \coloneqq \sum_{\substack{m \geq 1 \\ \N_{K/\mathbb{Q}}(\mathfrak{p})^m \leq x \\ \Frob_{\mathfrak{p}}^m = C}} \frac{1}{m},
\end{equation*}
a modified version of $\pi_C(x, L/K)$. Then it is clear that $\pi_C(x,L/K) \leq \widetilde{\pi}_C(x,L/K)$. See~\cite[Section~2.3]{Zyw} for elegant functorial properties enjoyed by the function $\widetilde{\pi}_C$.

\begin{lem}\label{lem:pi_f-to-chebotarev}
Let $f$ and $\ell$ be as above, let $\mathbb{Q} \subset L^B \subset L^H \subset L^U \subset L$ be constructed as above, and let $a$ be an integer. Let $q \geq 1$ be an integer with $\mathrm{gcd}(q,N_f\ell) = 1$, let $S \subseteq (\mathbb{Z}/q\mathbb{Z})^\times$ be a set of residue classes modulo $q$
, and suppose we have the following: for a prime number $p$, if $a_f(p) = a$, then $p \in S$.
\begin{enumerate}
    \item We have
    \begin{equation}\label{eq:pi_f-to-chebotarev-a}
        \pi_f(x,a;\ell) \leq \widetilde{\pi}_{C' \times S}(x,L^U(\zeta_q)/L^B) + 1.
    \end{equation}
    \item If $a \equiv 0 \Mod{\ell}$, we have
    \begin{equation}\label{eq:pi_f-to-chebotarev-b}
        \pi_f(x,0;\ell) \leq \widetilde{\pi}_{C'' \times S}(x,L^H(\zeta_q)/L^B) + 1.
    \end{equation}
\end{enumerate}
Note that if $q = 1$, then $\pi_f(x,a;\ell) \leq \widetilde{\pi}_{C'}(x,L^U/L^B)$ and $\pi_f(x,0;\ell) \leq \widetilde{\pi}_{C''}(x,L^H/L^B)$.
\end{lem}
\begin{proof}
From the condition $\mathrm{gcd}(q,N_f\ell) = 1$, the fields $L$ and $\mathbb{Q}(\zeta_q)$ are linearly disjoint, so $\Gal(L(\zeta_q)/\mathbb{Q}) = \Gal(L/\mathbb{Q}) \times \Gal(\mathbb{Q}(\zeta_q)/\mathbb{Q}) = G \times (\mathbb{Z}/q\mathbb{Z})^\times$. We claim that
\begin{equation*}
    \pi_f(x,a;\ell) \leq \pi_{C \times S}(x,L(\zeta_q)/\mathbb{Q}) + 1.
\end{equation*}
Indeed, let $p \leq x$ be a prime such that $p \nmid N_f \ell$, $a_f(p) = a$ and $a^2 - 4p^{k_f - 1}$ is a square in $\mathbb{F}_{\ell}^{\times}$. The representation $\rho_{f,\ell}$ is unramified at $p$ and we have $\tr(\rho_{f,\ell}(\Frob_p)) \equiv a_f(p) = a \Mod{\ell}$ and $\det(\rho_{f,\ell}(\Frob_p)) \equiv p^{k_f - 1} \Mod{\ell}$. Thus $\tr(\rho_{f,\ell}(\Frob_p))^2 - 4\det(\rho_{f,\ell}(\Frob_p)) \in \mathbb{F}_{\ell}$ is a square. We have thus shown that $\rho_{f,\ell}(\Frob_p) \subseteq C$. By hypothesis, we also know that if $a_f(p) = a$, then $p \in S \Mod{q}$.   Adding an extra $1$ to account for the excluded prime $p = \ell$, this implies that 
\begin{align*}
    \pi_f(x,a;\ell) &\leq \pi_{C \times S}(x,L(\zeta_q)/\mathbb{Q}) +1\\
    &\leq \widetilde{\pi}_{C \times S}(x,L(\zeta_q)/\mathbb{Q}) +1.
\end{align*}
By ~\cite[Lemma~2.6(i)]{Zyw}, and since $L^B$ is the fixed field in $L(\zeta_q)$ of the subgroup $B \times (\mathbb{Z}/q\mathbb{Z})^\times$, we have that 
    \begin{equation*}
    \pi_f(x, a; \ell) \leq \widetilde{\pi}_{(C \cap B) \times S}(x,L(\zeta_q)/L^B) +1. 
    \end{equation*}
    By ~\cite[Lemma~2.6(ii)]{Zyw}, and since $U \times \{1\}$ is a normal subgroup of $B \times (\mathbb{Z}/q\mathbb{Z})^\times$ with $(U \times\{1\}) \cdot ((C \cap B) \times (\mathbb{Z}/q\mathbb{Z})^\times) \subseteq ((C \cap B) \times (\mathbb{Z}/q\mathbb{Z})^\times)$, we also have that 
    \begin{equation*}
    \pi_f(x, a; \ell) \leq \widetilde{\pi}_{C' \times S}(x,L^U(\zeta_q)/L^B) +1,
\end{equation*}
which is part (a). If $a \equiv 0 \Mod{\ell}$, then $C$ consists of trace zero matrices, and so we repeat the last line with the subgroup $H \times \{1\}$ in the place of the subgroup $U \times \{1\}$.
\end{proof}

Now, for a Galois extension of number fields $L/K$, let $\mathcal{P}(L/K)$ denote the set of rational primes $p$ that are divisible by some prime ideal $\mathfrak{p}$ of $K$ that ramifies in $L$. Define
\begin{equation}\label{eq:M(L/K)}
    M(L/K) \coloneqq 2[L:K]D_K^{1/[K:\mathbb{Q}]} \cdot \prod_{p \in \mathcal{P}(L/K)} p.
\end{equation}
The next lemma transitions from the modified $\widetilde{\pi}_C$ to $\pi_C$, with explicit remainder term.

\begin{lem}\label{lem:tilde-to-no-tilde}
Let $L/K$ be a normal extension of number fields with Galois group $G$ and let $C \subseteq G$ be a union of conjugacy classes. If $x \geq 4$, then
\begin{align*}
    \widetilde{\pi}_C(x,L/K) &\leq \pi_C(x,L/K) + 3.15546 n_K \frac{x^{1/2}}{ \log x}  + n_K \log M(L/K).
\end{align*}
\end{lem}
\begin{proof}
The lemma follows from inspection of the proof of~\cite[Lemma~2.7]{Zyw}, using the identity $\sum_{m=2}^\infty m^{-2} = \frac{1}{6} \pi^2 - 1$ and the universal bound $\pi(x) < 1.25506\frac{x}{\log x}$.
\end{proof}

\subsection{Bounding $M(L/K)$} The combination of \cref{cor:pi_f-sieve}, \cref{lem:pi_f-to-chebotarev}, and \cref{lem:tilde-to-no-tilde} allows us to express $\pi_f(x, a)$ in terms of $\pi_{C' \times S}(x,L^U(\zeta_q)/L^B)$ or $\pi_{C' \times S}(x,L^H(\zeta_q)/L^B)$ if $a = 0$. Before proceeding, we want to establish bounds for the size of the field constants $M(L^U(\zeta_q)/L^B)$ and $M(L^H(\zeta_q)/L^B)$ in terms of the quantities $q$, $N_f$, and $\ell$. First we need the following result due to \cite{Ser81}.

\begin{lem}\label{lem: ser disc bound}
We have
\begin{equation}
    D_{L^B}^{1/[L^B:\mathbb{Q}]} \leq \rad(N_f)\ell \cdot (\ell+1)^{\omega(N_f)  + 1}.
\end{equation}
\end{lem}
\begin{proof}
An application of Proposition 6 of \cite{Ser81} gives
\begin{equation}
    \log D_{L^B} \leq \ell \sum_{p \in \mathcal{P}(L^B/\mathbb{Q})} \log  p  + |\mathcal{P}(L^B/\mathbb{Q})| (\ell + 1)  \log (\ell + 1).
\end{equation}
Since $p \in \mathcal{P}(L^B/\mathbb{Q})$ implies $p \mid N_f \ell$, we may bound the above sum by $\log (\rad(N_f)\ell)$. Similarly, we may bound $|\mathcal{P}(L^B/\mathbb{Q})|$ by $\omega(N_f) + 1$, proving the lemma.
\end{proof}

We use this result to prove the next lemma.

\begin{lem}\label{lem:M(L/K)-bound}
Suppose that $\ell > \ell_0$ with $\mathrm{gcd}(\ell -1,k_f - 1) = 1$, and $\mathbb{Q} \subset L^B \subset L^H \subset L^U \subset L$ are constructed as above, and $q \geq 1$ is an integer with $\mathrm{gcd}(q,N_f\ell) = 1$. Let $\omega(n)$ denote the number of distinct prime factors dividing a positive integer $n$, and let $\rad(n)$ denote their product, with $\rad(1) = 1$. We have that
\begin{equation}\label{eq:M(L/K)-bound-a}
    M(L^U(\zeta_q)/L^B) \leq 2 (\ell + 1)^{\omega(N_f) + 4} \rad(N_f)^2 \varphi(q)\rad(q) 
\end{equation}
and
\begin{equation}\label{eq:M(L/K)-bound-b}
    M(L^H(\zeta_q)/L^B) \leq 2 (\ell + 1)^{\omega(N_f) + 3} \rad(N_f)^2 \varphi(q) \rad(q).
\end{equation}
\end{lem}
\begin{proof}
Using the previous lemma with $[L^U(\zeta_q):L^B] = \varphi(q)(\ell - 1) \ell$ gives
\begin{align*}
    M(L^U(\zeta_q)/L^B) &= 2[L^U(\zeta_q):L^B] D_{L^B}^{1/[L^B:\mathbb{Q}]} \cdot \prod_{p \in \mathcal{P}(L^U(\zeta_q)/L^B)} p \\
    &\leq 2\varphi(q)(\ell - 1) \ell \cdot \rad(N_f) \ell \cdot (\ell + 1)^{\omega(N_f) + 1} \prod_{p \in \mathcal{P}(L^U(\zeta_q)/L^B)} p .
\end{align*}
Since $p \in \mathcal{P}(L^U(\zeta_q)/L^B)$ implies $p \mid q N_f \ell$, where $q$, $\ell$ and $N_f$ are assumed pairwise coprime, we may bound the product by $\rad(qN_f)\ell$. 

The proof for $M(L^H(\zeta_q)/L^B)$ is practically identical; the only change is due to the fact that $[L^H(\zeta_q):L^B] = \varphi(q)(\ell - 1)$.
\end{proof}

Observe that the bound \eqref{eq:M(L/K)-bound-b} that we achieve for $M(L^H(\zeta_q)/L^B)$ appears with one less power of $\ell + 1$ than \eqref{eq:M(L/K)-bound-a}. Because of this, working with $\pi_{C''}(x, L^H(\zeta_q)/L^B)$ will be advantageous when possible.

Finally, if $L/K$ is an abelian extension, we can relate $M(L/K)$ to the quantity $D_K \mathcal{Q}(L/K)$.

\begin{lem}\label{prop:M(L/K)-abelian}
Let $L/K$ be an abelian extension of number fields. Then
\begin{equation*}
    D_K \mathcal{Q}(L/K) \leq \Big( \frac{M(L/K)}{2} \Big)^{2[K:\mathbb{Q}]}.
\end{equation*}
\end{lem}
\begin{proof}
By~\cite[Proposition~2.5]{MMS}, we have
\begin{equation*}
    \mathcal{Q}(L/K) \leq \Big( [L:K] \prod_{p \in \mathcal{P}(L/K)} p \Big)^{2[K:\mathbb{Q}]}.
\end{equation*}
We multiply both sides by $D_K^2$ and use the definition \eqref{eq:M(L/K)} of $M(L/K)$.
\end{proof}

\subsection{Lang--Trotter type bounds}
Let $f(z)$ be a non-CM cusp form of even weight $k_f \geq 2$, level $N_f$, and trivial nebentypus with integer coefficients $a_f(n)$ which is an eigenform for all Hecke operators and Atkin--Lehner operators. Suppose furthermore that for some integer $q \geq 1$, $\mathrm{gcd}(q,\ell) = 1$ and subset $S \subseteq (\mathbb{Z}/q\mathbb{Z})^\times$, it is known that if $a_f(p) = a$, then $p \in S \Mod{q}$. Finally, suppose that the Galois representation $\rho_{f,\ell}$ is surjective for all primes $\ell > \ell_0$ with $\mathrm{gcd}(\ell-1,k_f-1) = 1$. For such $\ell$, constructing $L^B \subset L^H \subset L^U$ as before, we have the bounds for $M(L^U(\zeta_q)/L^B)$ and $M(L^H(\zeta_q)/L^B)$ provided in~\cref{lem:M(L/K)-bound}.

In this subsection, we are interested in estimating the prime counting function $\pi_f(x,a)$.  Of particular interest is $\pi_f(x, 0)$, as this quantity is related to the non-vanishing of the Fourier coefficients $a_f(n)$ of $f$ (see ~\cref{prop:D_f} below). First, using \cref{lem: BT ugly}, we develop an estimate for $\pi_{C'' \times S}(x,L^H(\zeta_q)/L^B)$ and $\pi_{C' \times S}(x,L^U(\zeta_q)/L^B)$ in terms of only $\ell$ and $q$.

\begin{lem}\label{lem: piCS bound}
Let $\ell > \max\{\ell_0,5\}$ be a prime with $\mathrm{gcd}(\ell - 1,k_f - 1) = 1$ and $q$ a positive integer such that $\gcd(q, \ell) = 1$.
Let
\begin{align*}
    c_{13} &= \max(62+4.2 \log\rad(N_f),\ 4.2(2.9 + \log \varphi(q))), \\
    c_{14} &= \max(42, 4.2(5.8 + \log \varphi(q) + \log \rad(q) + \log 2 ( 1 + \omega(q) + \omega(N_f)))), \\
    c_{15} &= 4.2(\omega(N_f) + 3.5).
\end{align*}
If
\begin{equation}\label{eq: lq range}
x \ge e^{c_{13}} e^{c_{14}(\ell + 1)} (\ell + 1)^{c_{15}(\ell + 1)},
\end{equation}
then
\begin{equation}
    \pi_{C'' \times S}(x,L^H(\zeta_q)/L^B) \le 11.29 \frac{|S|x}{\varphi(q) (\ell -1)\log x}.
\end{equation}
\end{lem}
\begin{proof}
We wish to express all the quantities present in \cref{lem: BT ugly} in terms of $\ell $ and $q$. For the extension $L^H(\zeta_q)/L^B$, we have that
\begin{align*}
    n_{L^B} &= \ell + 1, \\
    [L^H(\zeta_q):L^B] &= \varphi(q) (\ell - 1) \le \varphi(q) e^{\ell + 1 -3}.
\end{align*}
By Lemmas \ref{prop:M(L/K)-abelian} and \ref{lem:M(L/K)-bound}, we have that
\begin{equation}
    D_{L^B}\mathcal{Q}(L^H(\zeta_q)/L^B) \le (\ell+1)^{2(\omega(N_f)+3)(\ell + 1)} (\varphi(q)\rad(q))^{2(\ell + 1)}.
\end{equation}
By \cref{lem: ser disc bound}, we also have
\begin{equation}
    D_{L^B} \le (\rad(N_f)(\ell + 1))^{\ell + 1} (\ell+1)^{(\omega(N_f)+1) (\ell + 1)}.
\end{equation}
Lastly, to bound $Z_\m(\delta)$, note that the Artin conductor for $L^H(\zeta_q)/L^B$ is divisible only by primes in $L^B$ which ramify in $L^H(\zeta_q)$. By considering ramification of rational primes in $L$ and $\mathbb{Q}(\zeta_q)$ separately, $\mathcal{P}(L^H(\zeta_q)/L^B)$ is a subset of the primes in $L^B$ which lie over primes in $\mathcal{P}(L(\zeta_q)/\mathbb{Q})$, i.e. primes in $L^B$ which divide $qN_f \ell$. So, we obtain the bound 
\begin{equation*}
\omega(\mathfrak{m}) \leq \omega(q N_f \ell)[L^B:\mathbb{Q}] = (1 + \omega(q) + \omega(N_f))(\ell + 1),
\end{equation*}
which gives 
\begin{equation}
Z_\m(\delta) \le 2^{\omega(\mathfrak{m})} \leq 2^{(1 + \omega(q) + \omega(N_f))(\ell + 1)}.
\end{equation}
Now, applying these bounds to \cref{lem: BT ugly} with the parameters $\delta = 1/10$, $\eta = 1/21$, $\epsilon = \omega = 1/4$, and $\gamma = 1/5$ gives
\begin{align*}
    e^{\epsilon} \Big( e^{c_{11}} e^{c_{12} n_{L^B}} {n_{L^B}}^{n_{L^B}/2} D_{L^B}^{1/2} \Big)^{\frac{2}{(1-\eta)\omega}} &\le e^{62}e^{42(\ell+1)} \rad(N_f)^{4.2} (\ell + 1)^{4.2(\ell + 1)(\omega(N_f) + 3)} \\
    (252n_{L^B}[L^H(\zeta_q):{L^B}])^2  &\le e^{12}e^{4(\ell + 1)} \varphi(q)^2   \\
    \Big( \frac{  C(\delta, \epsilon, \gamma )}{1 - \delta - 2\omega(1 + \gamma)} \Big)^{\nu}  &\le\\
    \Big(e^{2.9}\varphi(q) (\varphi(q)\rad(q))^{\ell + 1}  e^{5.8(\ell + 1)}\Big.  & \Big.2^{(1 + \omega(q) + \omega(N_f))(\ell + 1)} (\ell + 1)^{(\omega(N_f) + 3.5)(\ell + 1)}\Big)^{4.2}.
\end{align*}
Taking the maximum of these three expressions gives the desired result.
\end{proof}

A similar result holds for $\pi_{C' \times S}(x,L^U(\zeta_q)/L^B)$.

\begin{lem}\label{lem: piCS U bound}
Let $\ell > \max\{\ell_0,5\}$ be a prime with $\mathrm{gcd}(\ell - 1,k_f - 1) = 1$ and $q$ a positive integer such that $\gcd(q, \ell) = 1$. 
Let
\begin{align*}
    c_{16} &= \max(62+4.2 \log\rad(N_f),\ 4.2(0.9 + \log \varphi(q))), \\
    c_{17} &= \max(42, 4.2(6.8 + \log \varphi(q) + \log \rad(q) + \log 2 ( 1 + \omega(q) + \omega(N_f)))), \\
    c_{18} &= 4.2(\omega(N_f) + 4.5).
\end{align*}
If
\begin{equation}\label{eq: lq U range}
x \ge e^{c_{16}} e^{c_{17}(\ell + 1)} (\ell + 1)^{c_{18}(\ell + 1)},
\end{equation}
then
\begin{equation}
    \pi_{C' \times S}(x,L^U(\zeta_q)/L^B) \le 11.29 \frac{|S|x}{\varphi(q) (\ell -1)\log x}.
\end{equation}
\end{lem}
\begin{proof}
The proof is the same as that of \cref{lem: piCS bound}, but we instead use the bound
\begin{equation}
    D_{L^B}\mathcal{Q}(L^U(\zeta_q)/L^B) \le (\ell+1)^{2(\omega(N_f)+3)(\ell + 1)} (\varphi(q)\rad(q))^{2(\ell + 1)}
\end{equation}
from Lemmas \ref{prop:M(L/K)-abelian} and \ref{lem:M(L/K)-bound} and
\begin{equation}
    [L^U(\zeta_q):L^B] = \varphi(q)\ell (\ell - 1) \le \varphi(q) e^{2(\ell + 1) -5}.
\end{equation}
\end{proof}
Note that if the range conditions \eqref{eq: lq range} or \eqref{eq: lq U range} hold with $\ell \ge 5$, we have that $x \ge e^{272}$, regardless of $N_f$ and $q$. Therefore, this will be established as an assumption throughout.

To conclude, we derive a general-purpose bound for $\pi_f(x, 0)$ and $\pi_f(x, a)$ for a newform $f$ with the aforementioned specifications. We can do so using Lemmas \ref{lem: piCS bound} or \ref{lem: piCS U bound} combined with \cref{cor:pi_f-sieve}, \cref{lem:pi_f-to-chebotarev}, and \cref{lem:tilde-to-no-tilde}.

\begin{thm}\label{thm: pi_f-main-lemma}
Let $a \in \ZZ$, $A > 272$ and $1 < m < 10$ be constants. Suppose $x \geq e^A$ and $\ell = \ell_1$ is a prime number, and set
\begin{equation}
    t = \Big\lceil \frac{2r}{(\log 2)}\log \log x \Big\rceil.
\end{equation}
Suppose that the following also hold:
\begin{enumerate}[label=(\Alph*)]
    \item \label{thm: pi_f-main-lemma-A} $\ell > \max\{\ell_0,5\}$ and $\mathrm{gcd}(\ell - 1,k_f - 1) = 1$ and $\mathrm{gcd}(\ell,q) = 1$,
    \item \label{thm: pi_f-main-lemma-B} $x \ge e^{c_{16}} e^{c_{17}(\ell + 1)} (\ell + 1)^{c_{18}(\ell + 1)}$, where $c_{16}$, $c_{17}$, $c_{18}$ are as in \eqref{eq: lq U range}. If $a = 0$, we may instead take $x \ge e^{c_{13}} e^{c_{14}(\ell + 1)} (\ell + 1)^{c_{15}(\ell + 1)}$, where $c_{13}$, $c_{14}$, $c_{15}$ are as in \eqref{eq: lq range}.
    \item \label{thm: pi_f-main-lemma-C}$\ell < \frac{1}{2}\exp( \frac{\log x}{2t})$,
    \item \label{thm: pi_f-main-lemma-D} There exist $t$ primes $\ell = \ell_1 < \ell_2 < \cdots < \ell_t$ in the interval $[\ell, 2\ell)$ such that $\gcd(\ell_j - 1, k_f - 1) = 1$ and $\gcd(\ell_j, q) = 1$.
\end{enumerate}
We have that
\begin{align} \label{eq:pi_f-main-lemma}
    \pi_{f}(x, a) &\leq 34.7r \frac{|S|}{\varphi(q) (\ell- 1)} \frac{x \log \log x}{\log x}
    +  1.42 \frac{x}{(\log x)^r}.
\end{align}
\end{thm}
\begin{proof}
We apply~\cref{lem:pi_f-to-chebotarev}(a) along with \cref{lem:tilde-to-no-tilde}. Using $[L^B:\mathbb{Q}] = \ell + 1$ and \eqref{eq:M(L/K)-bound-a} to bound for $M(L^U(\zeta_q)/L^B)$ gives
\begin{align}
    \pi_f(x,a;\ell) &\leq \pi_{C' \times S}(x,L^U(\zeta_q)/L^B) + 3.15546 \ell(\ell + 1) \frac{x^{1/2}}{\log x} \nonumber\\ &\phantom{\,\,\,\,\,\,\,\,}+ (\omega(N_f) + 4)(\ell + 1) \log (\ell + 1) + (\ell + 1) \log (2\rad(N_f)^2\varphi(q)\rad(q)) + 1.
\end{align}
If $a = 0$, we instead use \cref{lem:pi_f-to-chebotarev}(b), and \eqref{eq:M(L/K)-bound-b} to bound $M(L^H(\zeta_q)/L^B)$
. We obtain a similar bound for $\pi_f(x, 0 ;\ell)$, this time with a main term of $\pi_{C'' \times S}(x, L^H(\zeta_q)/L^B)$.
Under the range condition \ref{thm: pi_f-main-lemma-B}, the sum of the error terms above will be less than $0.01\frac{x}{\varphi(q)(\ell_1 - 1) \log x}$. Thus we have that
\begin{align}\label{eq:pi_f-to-chebotarev}
    \pi_f(x,a;\ell) &\leq \pi_{C' \times S}(x,L^U(\zeta_q)/L^B) + 0.01 \frac{x}{\varphi(q) (\ell_1 -1) \log x}.
\end{align}
The analogous equation holds when $a = 0$. Since $x$ and $\ell_1$ satisfy the preconditions of~\cref{lem: piCS U bound}, we have
\begin{equation}\label{eq:applying-chebotarev}
    \pi_{C'' \times S}(x,L^H(\zeta_q)/L^B) \leq   \frac{11.29 |S|x}{\varphi(q)(\ell_1 - 1)\log x}.
\end{equation}
If $a = 0$, we instead use \cref{lem: piCS bound}. \eqref{eq:pi_f-to-chebotarev} and \eqref{eq:applying-chebotarev} then give
\begin{align}\label{eq:pi_f-ell-bound}
    \pi_{f}(x, a; \ell) &\leq \frac{11.3 |S|x}{\varphi(q)(\ell_1 - 1)\log x}
\end{align}
for all primes $\ell$ satisfying conditions~\ref{thm: pi_f-main-lemma-A} and ~\ref{thm: pi_f-main-lemma-B} of the theorem statement.

Let $\ell_1 < \ell_2 < \cdots < \ell_t$ be the $t$ primes determined in condition \ref{thm: pi_f-main-lemma-D} of the theorem statement. Due to \ref{thm: pi_f-main-lemma-C} and \ref{thm: pi_f-main-lemma-D}, the preconditions of \cref{cor:pi_f-sieve} hold so we have
\begin{align}\label{eq:applying-cor-a}
    \pi_{f}(x,a) &\leq \Big( \frac{2r}{\log 2} \log \log x + 1 \Big)\Big(\max_{1 \leq j \leq t}\pi_{f}(x,a;\ell_j) \Big) \nonumber\\&\phantom{\,\,\,\,\,\,\,\,}+ \frac{x\sqrt{2}}{(\log x)^r} + \frac{2x^{1/2}}{(\log x)^r} + 2 (\log x)^{2r} + 35. 
\end{align}
Therefore, using that $x \ge e^{272}$ from \eqref{eq: lq range}, we can combine error terms to obtain that 
\begin{equation}\label{eq:applying-cor-b}
    \pi_{f}(x,a) \leq \Big(\frac{2r}{\log 2} + \frac{1}{\log 272}\Big) \log \log x \Big( \max_{1 \leq j \leq t} \pi_{f}(x,a;\ell_j) \Big) + 1.42 \frac{x}{(\log x)^r}.
\end{equation}

Finally, we combine \eqref{eq:applying-cor-b} with the estimates for $\pi_{f}(x,0;\ell_j)$ provided in \eqref{eq:pi_f-ell-bound}, and in order to take the maximum we note that the right-hand side of \eqref{eq:pi_f-ell-bound} is maximized over the $\ell_j$'s when $ j= 1$.
\end{proof}

In the next sections, we make \cref{thm: pi_f-main-lemma} explicit in two different cases, proving Theorems \ref{thm: LT tau bound} and \ref{thm: level 2 bound}. 

\begin{remark}\label{rem: general lt}
\cref{thm: pi_f-main-lemma} can be made explicit for other non-CM cusp forms $f$ if the surjectivity of the $\ell$-adic representation $\rho_{f, \ell}$ is known beyond some threshold $\ell > \ell_0$. Explicit values of $\ell_0$ are only known for certain forms, such as the newforms of level 1 and weights $k\in \{12,16,18,20,22,26\}$.

We also require bounds for the number of primes in the interval $[\ell, 2\ell]$ such that 
$\gcd(\ell -1, k_f -1 ) = 1$; we supply such bounds for $k_f = 12$ in \cref{lem: primes in interval}.
\end{remark}

%% file: 8-LT-applications.tex
We first specialize to the delta function 
$f(z) = \Delta(z)$ defined in \eqref{eq: deltadefinition}. Recall that $\Delta$ is a cuspidal newform of weight $k_{\Delta} = 12$ and level $N_{\Delta} = 1$ with its $n$-th Fourier coefficient equal to Ramanujan's tau-function $\tau(n)$. By work of Swinnerton-Dyer ~\cite[Theorem 4, Corollary (i)]{Swi}, if $\ell > 691$ is prime, then the associated residual representation $\rho_{\Delta,\ell}$ is surjective. 
Since $N_{\Delta}$ has no prime factors, we have that $\omega(N_{\Delta}) = 0$ and $\rad(N_{\Delta}) = 1$. Moreover, Serre~\cite{Ser73} proved that if $\tau(p) = 0$, then $p$ lies in one of $|S| = 33$ congruence classes modulo 
\begin{align*}
    q &\coloneqq 3488033912832000 \\
    & = 2^{14} \times 3^7 \times 5^3 \times 7^2 \times 23 \times 691.
\end{align*}
We apply this data to the statement of \cref{thm: pi_f-main-lemma}. First, we require bounds for functions which count rational primes.

\subsection{Auxiliary results on bounds for prime-counting functions}
We first state a version of the classical Brun--Titchmarsh theorem due to Montgomery--Vaughan~\cite{MV}:
\begin{thm}\label{thm: mvbt}
Let $x$ and $y$ be positive real numbers, and let $d$ and $q$ be relatively prime positive integers. If $\pi(x;q,d) = \#\{p \leq x \mid p \equiv d \Mod{q}\}$, then 
\begin{equation*}
    \pi(x+y;q,d) - \pi(x;q,d) < \frac{2y}{\varphi(q) \log (y/q)}.
\end{equation*}
\end{thm}
We also need an estimate for the number of primes between $y $ and $2y$ from \cite{rosser1962}.
\begin{lem}[Rossen--Schoenfeld\protect{\cite[Corollary~3]{rosser1962}}]\label{lem: bertrand}
For all $y \ge 20.5$, we have
\begin{equation}
    \pi(2y) - \pi(y) > 0.6 \frac{y}{\log y}
\end{equation}
\end{lem}

Combining these results gives the following lemma, which we use to verify conditions \ref{thm: pi_f-main-lemma-C} and \ref{thm: pi_f-main-lemma-D} of \cref{thm: pi_f-main-lemma}.

\begin{lem}\label{lem: primes in interval}
Let $y \ge 2000$. Then
\begin{equation}\label{eq: primes in interval}
    \#\{ p \in (y, 2y]\mid  p \not\equiv 1\Mod{11}\} > 0.3 \frac{y}{\log y}.
\end{equation}
\end{lem}
\begin{proof}
The count on the left-hand side of \eqref{eq: primes in interval} is equal to
\begin{equation*}
    \pi(2y) - \pi(y) - (\pi(2y;11,1) - \pi(y;11,1)).
\end{equation*}
By Theorem \ref{thm: mvbt} and the assumption that $y \ge 2000$, we bound $\pi(2y;11,1) - \pi(y;11,1)$ by
\begin{equation}
    \pi(2y;11,1) - \pi(y;11,1) \le 0.2\frac{ y}{\log y - \log 11} \le 0.3 \frac{y}{\log y}.
\end{equation}
Combining this estimate with \cref{lem: bertrand} gives
\begin{align}
       \pi(2y) - \pi(y) - (\pi(2y;11,1) - \pi(y;11,1)) &> (0.6-0.3) \frac{y}{\log y} = 0.3 \frac{y}{\log y}
\end{align}
as desired.
\end{proof}

\subsection{Proof of \cref{thm: LT tau bound}}\label{sec: LT tau bound}
We apply \cref{thm: pi_f-main-lemma}, choosing $\ell_1$ and $m$ so that the conditions of the theorem hold. For general $a \in \ZZ$, we take $q = 1$ and recall that $N_f = 1$ so that the range condition \ref{thm: pi_f-main-lemma-B} becomes
\begin{equation}\label{eq: general tau range}
    x \ge e^{42} e^{62(\ell + 1)} (\ell + 1)^{18.9(\ell + 1)}.
\end{equation}
If $a = 0$, we instead use $ q = 3488033912832000$ and the alternate range in condition \ref{thm: pi_f-main-lemma-B}, so we have
\begin{equation}\label{eq: 0 tau range}
    x \ge e^{156} e^{252(\ell + 1)} (\ell + 1)^{14.7(\ell + 1)}.
\end{equation}
Letting $\theta > 0$ be some fixed parameter to be specified shortly, we will choose the prime $\ell_1$ to be near (but smaller than) the function
\begin{equation}
    \ell(x) \coloneqq \frac{\theta \log x}{\log \log x}.
\end{equation}
We take $\theta = 0.06$ for general $a \in \mathbb{Z}$, and $\theta = 0.04$ when $a = 0$. When $x \ge e^{e^{16}}$, we can verify that \eqref{eq: general tau range} holds with $\ell = \ell(x)$ and $\theta = 0.06$ and that \eqref{eq: general tau range} holds with $\theta = 0.04$. We can also verify that condition \ref{thm: pi_f-main-lemma-C} holds.

More precisely, we set $\ell_1$ to be the largest prime number less than $\ell(x)$ which is not congruent to $1$ modulo $11$. By \cref{lem: primes in interval}, we can take $\ell_1 > 0.5 \ell(x)$. Now, setting $r = 4$, we use \cref{lem: primes in interval} to verify that condition \ref{thm: pi_f-main-lemma-C} holds when $x \ge e^{e^{16}}$. Now, using $\theta = 0.06$, $r =4$, and $q = 1$ gives
\begin{equation}
    \pi_{\Delta}(x,a) \leq 4626.7 \frac{x (\log \log x)^2}{(\log x)^2} + 1.42 \frac{x}{(\log x)^4}.
\end{equation}
Using $\theta = 0.04$, $r=4$, and $q = 3488033912832000$ likewise gives

\begin{equation}
    \pi_{\Delta}(x,0) \leq 3.007 \cdot 10^{-10} \frac{x (\log \log x)^2}{(\log x)^2}  + 1.42 \frac{x}{(\log x)^4}.
\end{equation}

If $x \ge e^{e^{16}}$, then 
\begin{equation}
    1.42 \frac{x}{(\log x)^4} \le 10^{-16} \frac{x (\log \log x)^2}{(\log x)^2}.
\end{equation}
Absorbing the error term into the main term yields the desired result. \qed

Next, we use \cref{thm: LT tau bound} to prove \cref{thm: Lehmer+}.

\subsection{Proof of \cref{thm: Lehmer+}}\label{sec: Lehmer}

We compute a lower bound for 
\begin{equation*}
    D_{\Delta} = \lim_{x \rightarrow \infty} \frac{\#\{1 \leq n \leq x \mid \tau(n) \neq 0\}}{x},
\end{equation*}
using our result for $\pi_{\Delta}(x,0)$. These two quantities are related by the following result. 

\begin{prop}\label{prop:D_f}
For all $X_0 \geq 2$, we have that
\begin{align}
    D_{\Delta} &= \prod_{\tau(p) = 0} \left( 1 - \frac{1}{p+1} \right) 
    \label{eq:D_f-line-a} \\ &> \exp\left( - \int_{X_0}^\infty \frac{\pi_\Delta(x, 0)}{x(x+1)} \, dx \right) \cdot \prod_{\substack{p \leq X_0 \\ \tau(p) = 0}} \left( 1 - \frac{1}{p+1} \right) \label{eq:D_f-line-b}.
\end{align}
\end{prop}
\begin{proof}
The first line follows from~\cite[Equation~202]{Ser81}. 
The second line follows from a partial summation argument applied to the logarithm of the product in \eqref{eq:D_f-line-a}, and the identity for the truncated integral
\begin{equation*}
    \exp\left( - \int_{2}^{X_0} \frac{\pi_\Delta(x, 0)}{x(x+1)} \, dx \right) = \left( 1 + \frac{1}{X_0} \right)^m \prod_{j = 1}^m \left( 1 - \frac{1}{p_j + 1} \right)
\end{equation*}
where $p_1,\dots,p_j$ are the primes less than $X_0$ with $\tau(p) = 0$.
\end{proof}

\begin{proof}[Proof of \cref{thm: Lehmer+}]
We use \cref{prop:D_f} with the cutoff $X_0 = 10^{23}$. By computer search, Rouse and Thorner ~\cite{RT} 
showed that all but $1810$ prime numbers $p$ less than $10^{23}$ satisfy $\tau(p) \ne 0$, using Swinnerton-Dyer's congruences for $\tau(n)$ and Galois representations at $\ell = 11, 13, 17, 19$, which are necessary conditions for $\tau(p) = 0$ \cite{Swi}. They compute
\begin{equation}\label{eq: tau low range}
    \prod_{\substack{\tau(p) = 0 \\ p \leq 10^{23}}}\left(1 - \frac{1}{p+1}\right) > 0.99999999999999999980399.
\end{equation}
From the fact that if $\tau(p) = 0$, then $p$ lies in one of $33$ possible residue classes modulo $q$,~\cref{thm: mvbt} allows us to bound
\begin{equation}\label{eq: middle range prime bound}
    \pi_\Delta(x, 0) \leq 1810 + \frac{2(x-10^{23}+2q)}{\varphi(q)\log((x-10^{23}+2q)/q)} \times 33, \quad x > 10^{23}.
\end{equation}
To bound the contribution in this range, we use the following lemma.
\begin{lem}\label{lem: middle int bound}
If $X_0 > 2q > 0$, then
\begin{equation*}
   \int_{X_0}^{X_1} \frac{x - X_0 + 2q}{x(x+1)\log((x-X_0+2q)/q)} dx < \int_{2}^{X_1/q} \frac{x}{(x+X_0/q-2)^2 \log x} dx.
\end{equation*}
\end{lem}
\begin{proof}
We use that $X_0-2q > 0$, $x(x+1) > x^2$ when $x > 0$, and set $u = x-X_0+2q$ to find that
\begin{equation*}
    \int_{X_0}^{X_1} \frac{x - X_0 + 2q}{x(x+1)\log((x-X_0+2q)/q)} dx < \int_{2q}^{X_1} \frac{u}{(u+X_0-2q)^2\log(u/q)} du.
\end{equation*}
Now, setting $t = u/q$ and simplifying completes the proof.
\end{proof}
We want to bound the integral in \eqref{eq:D_f-line-b} between $X_0 = 10^{23}$ and $X_1 = e^{e^{16}}$. \eqref{eq: middle range prime bound} and \cref{lem: middle int bound} allow us to numerically integrate to find that
\begin{align*}
    \int_{X_0}^{X_1} \frac{\pi_\Delta(x, 0)}{x(x+1)}\, dx &< \int_{X_0}^{X_1} \frac{1810}{x(x+1)}\, dx + \frac{66}{\varphi(q)}\int_2^{X_1/q} \frac{x}{(x+X_0/q - 2)^2 \log x} \, dx \\
    &\le 1.1358 \times 10^{-12}.
\end{align*}
In the range $x > X_1$, we upper bound $\pi_\Delta(x, 0)$ using~\cref{thm: LT tau bound}, which gives
\begin{align*}
    \int_{X_1}^\infty \frac{\pi_\Delta(x, 0)}{x(x+1)} \,dx &\le (3.01 \cdot 10^{-10}) \int_{X_1}^\infty \frac{(\log \log x)^2}{x(\log x)^2} \,dx\\
    &= (3.01 \cdot 10^{-10})\cdot \frac{(\log \log X_1)^2 + 2 \log \log X_1 + 2}{\log X_1} \\
    & < 9.824 \cdot 10^{-15}.
\end{align*}
Thus we have that
\begin{align}
   \exp \left(-\int_{X_0}^\infty \frac{\pi_\Delta(x,0)}{x(x+1)}\,dx\right) &> \exp(-1.1358 \times 10^{-12} -9.824 \cdot 10^{-15}) \nonumber \\ &> 0.999999999998854 \label{eq: tau upper range} .
\end{align}
Combining the estimates \eqref{eq: tau low range} and \eqref{eq: tau upper range} via \cref{prop:D_f} yields
\begin{align*}
    D_\Delta &> 0.99999999999999999980399 \cdot 0.999999999998854 \\
    &> 0.99999999999885 \\
    &= 1 - 1.15 \cdot 10^{-12},
\end{align*}
as desired.
\end{proof}

%% file: 9-elliptic-curve.tex
Let $E$ be the elliptic curve defined by \eqref{eq: elliptic curve}. We apply our method from \cref{sec: nonzero} to the modular form $f_E(z)$ associated to $E$, as defined in \eqref{eq: E mod form}, in order to bound $\pi_E(x, a)$ and $D_{f_E}$. We first note the following congruence relation.
\begin{lem}\label{lem: ellipticcurvecongruence}
If $p \ne 11$, then $ a_E(p) \equiv p + 1 \mod 5$.
\end{lem}
\begin{proof}
The elliptic curve $E$ satisfies $E[5] \cong \ZZ/5\ZZ$ \cite[p. 55]{lang2006frobenius}. If $p \ne 11$, then $E$ has good reduction at $p$, so the reduction mod $p$ map $E[5] \to E(\mathbb{F}_p)$ is injective \cite[Proposition VII.3.1]{Sil09}. Then, $\#E(\mathbb{F}_p) \equiv 0 \mod 5$, and by the modularity theorem $a_E(p) = p + 1 - \#E(\mathbb{F}_p)$, so we have that $p + 1 - a_E(p) \equiv 0 \mod 5$. 
\end{proof}
This allows us to set $q = 5$ below. Note that if $a \equiv 1 \mod 5$ and $a_E(p) = a$, then $p = 5$ or $11$. If $a \ne 1$ and $a \equiv {1}\mod{5}$, then $\pi_E(x, a) = 0$ since $a_E(5) = a_E(11) = 1$. For other values of $a$, we use \cref{thm: level 2 bound}.

\begin{proof}[Proof of \cref{thm: level 2 bound}]
We use \cref{thm: pi_f-main-lemma} as in the proof of \cref{thm: LT tau bound} in \cref{sec: LT tau bound}. Substituting $q = 5$ and $N_f = 11$ into \ref{thm: pi_f-main-lemma-B} yields the range condition
\begin{equation}
    x \ge e^{72} e^{46(\ell + 1)} (\ell + 1)^{18.9(\ell + 1)}.
\end{equation}
We repeat the arguments of \cref{sec: LT tau bound}, verifying that the conditions of \cref{thm: pi_f-main-lemma} hold with $\theta = 0.055$, $r = 2$ and $x \ge e^{e^{13}}$. This implies that
\begin{equation}\label{eq: piE main term}
    \pi_E(x, a) \le 630.91 \cdot \frac{x(\log \log x)^2}{(\log x)^2} + 1.42 \cdot \frac{x}{(\log x)^2}.
\end{equation}
If $x \ge e^{e^{13}}$, then 
\begin{equation}\label{eq: piE error term}
    1.42 \frac{x}{(\log x)^2} \le 0.01 \cdot \frac{x(\log \log x)^2}{(\log x)^2}.
\end{equation}
Substituting \eqref{eq: piE main term} into \eqref{eq: piE error term} completes the proof.
\end{proof}
We now prove \cref{thm: level 2}.

\begin{proof}[Proof of \cref{thm: level 2}]
We argue as in the proof of \cref{thm: Lehmer+} in \cref{sec: Lehmer}. By a similar argument as in the proof of \cref{prop:D_f} (applying \cite[Equation~201]{Ser81} instead of \cite[Equation~202]{Ser81}), we first note that
\begin{align}
D_{f_E} &= \frac{14}{15}\prod_{a_E(p) = 0}\left(1 - \frac{1}{p+1}\right) \nonumber\\ &> \frac{14}{15}\exp\left(-\int_{X_0}^\infty \frac{\pi_E(x, 0)}{x(x+1)}~dx\right) \cdot \prod_{\substack{p \leq X_0 \\ a_E(p) = 0}}\left(1 - \frac{1}{p+1}\right). \label{eq: DfE serre}
\end{align}
We use the cutoff $X_0 = 10^{11}$. By computer search, Rouse and Thorner \cite{RT}
find that there are precisely 17857 primes $p \leq 10^{11}$ such that $a_E(p) = 0$, and they compute that
\begin{equation}\label{eq: level 2 lower range}
    \frac{14}{15} \prod_{\substack{p \leq 10^{11} \\ a_E(p) = 0}}\left(1 - \frac{1}{p+1}\right) = 0.8465247961...
\end{equation}
If $x > 10^{11}$, then \cref{lem: ellipticcurvecongruence} and \cref{thm: mvbt} imply that
\begin{equation}\label{eq: elliptic curve bt}
    \pi_E(x, 0) \leq 17857 + \frac{2(x-10^{11}+10)}{4\log((x -10^{11}+10)/5)}.
\end{equation}
We want to bound the integral in \eqref{eq: DfE serre} between $X_0 = 10^{11}$ and $X_1 = e^{e^{13}}$. Applying \eqref{eq: elliptic curve bt}, \cref{lem: middle int bound} and integrating numerically gives the bound
\begin{equation*}
    \int_{X_0}^{X_1} \frac{17857 + \frac{2(x-10^{11}+10)}{4\log((x-10^{11}+10)/5)}}{x(x+1)}~dx < 4.898.
\end{equation*}
When $x > X_1$, we upper bound $\pi_E(x, 0)$ using \cref{thm: level 2 bound}, which yields
\begin{align*}
    \int_{X_1}^\infty \frac{\pi_E(x, 0)}{x^2 + x} dx &\le 631 \int_{X_1}^\infty \frac{(\log \log x)^2}{x(\log x)^2} dx < 0.281.
\end{align*}
We then have that
\begin{equation}\label{eq: level 2 upper range}
   \exp \left(-\int_{10^{11}}^\infty \frac{\pi_E(x,0)}{x(x+1)}~dx\right) > \exp(-4.898-0.281) > 0.00563364.
\end{equation}
Combining \eqref{eq: level 2 lower range} and \eqref{eq: level 2 upper range} via \eqref{eq: DfE serre} completes the proof.
\end{proof}